\documentclass[12pt]{article}
\usepackage{amsmath,amsfonts}
\usepackage{verbatim}
\usepackage{latexsym}
\usepackage{graphicx}
\usepackage{color}
\usepackage{slashed}
\makeatletter
\@addtoreset{equation}{section}

\begin{document}

\newcommand{\E}{\mathbb{E}}
\newcommand{\PP}{\mathbb{P}}
\newcommand{\CP}{\mathcal{P}}
\newcommand{\CU}{\mathcal{U}}
\newcommand{\CW}{\mathcal{W}}
\newcommand{\CK}{\mathcal{K}}
\newcommand{\RR}{\mathbb{R}}
\newcommand{\LL}{\mathbb{L}}
\newcommand{\HH}{\mathbb{H}}
\newcommand{\CL}{\mathcal{L}}
\newcommand{\NN}{\mathbb{N}}
\newcommand{\CC}{\mathcal{C}}
\newcommand{\BB}{\mathcal{B}}
\newcommand{\SM}{\mathbb{S}}
\newcommand{\CM}{\mathcal{M}}
\newcommand{\II}{\mathbf{1}}
\newcommand{\hX}{\hat X}
\newcommand{\bq}{\bar q}
\newcommand{\bV}{\bar V}

\newtheorem{theorem}{Theorem}[section]
\newtheorem{lemma}[theorem]{Lemma}
\newtheorem{coro}[theorem]{Corollary}
\newtheorem{defn}[theorem]{Definition}
\newtheorem{assp}[theorem]{Assumption}
\newtheorem{expl}[theorem]{Example}
\newtheorem{prop}[theorem]{Proposition}
\newtheorem{rmk}[theorem]{Remark}

\newcommand\tq{{\scriptstyle{3\over 4 }\scriptstyle}}
\newcommand\qua{{\scriptstyle{1\over 4 }\scriptstyle}}
\newcommand\hf{{\textstyle{1\over 2 }\displaystyle}}
\newcommand\hhf{{\scriptstyle{1\over 2 }\scriptstyle}}

\newcommand{\proof}{\noindent {\it Proof}. }
\newcommand{\eproof}{\hfill $\Box$} 

\def\a{\alpha} \def\g{\gamma} \def\nn{\nonumber}
\def\e{\varepsilon} \def\z{\zeta} \def\y{\eta} \def\o{\theta}
\def\vo{\vartheta} \def\k{\kappa} \def\lbd{\lambda} \def\m{\mu} \def\n{\nu}
\def\x{\xi}  \def\r{\rho} \def\s{\sigma}
\def\p{\phi} \def\f{\varphi}   \def\w{\omega}
\def\q{\surd} \def\i{\bot} \def\h{\forall} \def\j{\emptyset}

\def\be{\beta} \def\de{\delta} \def\up{\upsilon} \def\eq{\equiv}
\def\ve{\vee} \def\we{\wedge}

\def\d{\mathrm{d}}
\def\F{{\cal F}}
\def\T{\tau} \def\G{\Gamma}  \def\D{\Delta} \def\O{\Theta} \def\L{\Lambda}
\def\X{\Xi} \def\S{\Sigma} \def\W{\Omega}
\def\M{\partial} \def\N{\nabla} \def\Ex{\exists} \def\K{\times}
\def\V{\bigvee} \def\U{\bigwedge}

\def\1{\oslash} \def\2{\oplus} \def\3{\otimes} \def\4{\ominus}
\def\5{\circ} \def\6{\odot} \def\7{\backslash} \def\8{\infty}
\def\9{\bigcap} \def\0{\bigcup} \def\+{\pm} \def\-{\mp}
\def\la{\langle} \def\ra{\rangle}

\def\tl{\tilde}
\def\trace{\hbox{\rm trace}}
\def\diag{\hbox{\rm diag}}
\def\for{\quad\hbox{for }}
\def\refer{\hangindent=0.3in\hangafter=1}

\newcommand\wD{\widehat{\D}}

\title{
\bf
Hybrid Stochastic Functional Differential Equations with Infinite Delay: Approximations and Numerics
 }

\author{
{\bf
Guozhen Li ${}^{1}$,  
Xiaoyue Li ${}^{2}$,  
Xuerong Mao${}^{3,}$\thanks{Corresponding author. E-mail:  x.mao@strath.ac.uk} ,
Guoting Song${}^{4}$   
 }
\\
${}^1$ School of Mathematics and Statistics, \\
Northeast Normal University, Changchun, Jilin 130024, China. \\
${}^2$ School of Mathematical Sciences,\\ Tiangong University, Tianjin, 300387, China.\\
${}^3$ Department of Mathematics and Statistics, \\
University of Strathclyde, Glasgow G1 1XH, U.K. \\
${}^4$ Academy of Mathematics and Systems Science, \\Chinese Academy of Sciences, Beijing, 100190, China.
}

\date{}

\maketitle

\begin{abstract}
 This paper is to investigate if the solution of
a hybrid stochastic functional differential equation (SFDE) with infinite  delay can be approximated by the solution of the corresponding hybrid SFDE with finite delay. A positive result is established for a large class of highly nonlinear hybrid SFDEs with infinite delay.
Our new theory makes it possible to numerically approximate the solution of the hybrid SFDE
with infinite delay, via the numerical solution of the corresponding hybrid SFDE
with finite delay.


\end{abstract}

\section{Introduction}

Many dynamic systems in sciences and industry do not only depend on their current state but also past states due to unavoidable time delays, while they are often subject to various system parameter uncertainties and environmental noise.  Moreover,
random switching takes place frequently  in a finite set resulting in the systems being hybrid, in which continuous dynamics and discrete events coexist and interact.
Hybrid stochastic functional differential equations (SFDEs) have been used to model such dynamic systems.
There are two categories : (A) hybrid SFDEs with finite delays; (B) hybrid SFDEs with infinite delays.  There is a huge literature on type-(A) SFDEs  (see, e.g., \cite{Du21, FLM16, HM10,MY06, Nguy20, SM18, XLMZ05})  but much less on (B).  This is certainly not because
type-(B) SFDEs  have no use in applications but due to the fact that
it is much harder to study (B) than (A).  As a matter of fact, long memory,
also called long-range dependence or persistence, is a phenomenon that occurs in many fields including ecology,
biology, econometrics, linguistics, hydrology, climate, DNA sequencing.
It is due to the demand from the research in these fields, research on
type-(B) SFDEs  has been advanced quickly   (see e.g., \cite{Liu15, LW12, WWMZ20, WW07, WYM17}).
 However, there are many open problems to be solved in order to meet the need of applications.

 One of the open problems is: how to obtain numerical solutions to type-(B) SFDEs?  By numerical solutions we means they can be computed by computers but not theoretical approximate solutions which are not computable.  In this paper we will tackle this open problem by bridging the gap
 between
  type-(B) SFDEs  and type-(A) SFDEs.  More precisely, we will establish new approximation theory  between type-(B)  and type-(A) SFDEs.
 Based on our new theory,  type-(B) SFDEs can be approximated by the corresponding type-(A) SFDEs.  Applying numerical methods to type-(A) SFDEs, we can obtain the numerical solutions to the corresponding  type-(A) SFDEs.   By this bridge,  we also yield the numerical solutions to  type-(B) SFDEs.  Our approach can be illustrated as follows:

 $$
 \begin{array}{ccccc}
 \hbox{type-(B) SFDEs} &  \approx & \hbox{type-(A) SFDEs} & \Leftarrow &\hbox{numerical methods} \\
 \end{array}
 $$

 To see our approach does not only work but is also useful, we need to make four points clear:

 \begin{itemize}

 \item There
 are existing numerical methods on type-(A) SFDEs (see, e.g., \cite{WMK11, WMS10}), though the numerical theory in this area is still developing.
 In detail, Li and Hou \cite{LH06} discussed the Euler-Maruyama (EM) method for linear hybrid stochastic delay differential equations (SDDEs). Wu and Mao \cite{WM08} established the strong mean square convergence  of the EM method for neutral SFDEs
under the linear growth condition.
Recently, numerical methods for superlinear type-(A) SFDEs have been developing quickly.  For examples, under the generalized Khasminskii-type condition in terms of Lyapunov functions,
 Li et al. \cite{LMS10} proved that the EM numerical solutions converge to the exact ones in probability in any finite interval;
 Guo et al. \cite{GMY18} and Song et al. \cite{SHGL22} obtained the strong convergence of the truncated EM numerical solutions for type-(A) SDDEs;
 Zhang et al. \cite{ZSL18} extended the truncated EM method  to type-(A) SFDEs;
Dareiotis et al.  \cite{Sa16} extended the tamed EM to type-(A) SDDEs driven by L$\acute{e}$vy noise.

\item There are a number of papers where type-(B) SFDEs have been used to model population systems  \cite{Liu15, LW12}. To illustrate their results, the Milstein and EM methods were used respectively to perform some computer simulations but there is no explanation on whether the numerical methods are applicable to their superlinear type-(B) SFDEs.  In \cite{MCFM21}, the authors used the simulations on
  type-(A) SFDEs to illustrate the results on type-(B) SFDEs but once again these is lack of theoretical support.
  We hence see there is an urgent need to rigorously establish the new approximation theory  between type-(B)  and type-(A) SFDEs.

  \item There are a couple of papers where discrete-time solutions were proposed to re-produce the
  stability of some very special type-(B) SFDEs.  For example, Asker \cite{Asker21}
 presented the EM solutions of a neutral SFDE with infinite delays under the globally Lipschitz condition and examined the stability in distribution of numerical solutions.  It is noted that
 the stability analysis is theoretical and does not need to compute the EM solutions numerically.
 In fact, the EM solutions there are not computable as they need infinitely-many discrete-time initial data which any computer could not cope with.  In other words, their EM solutions are in fact only approximate ones.

 \item Similar situation has happened to the stability analysis of numerical methods for \emph{deterministic} functional differential equations with infinite delays.  Song and Baker \cite{SongBaker04} discretized  the \emph{deterministic} Volterra integro-differential equation with infinite delays  using the $\theta$ method and  proved that for a \emph{small bounded initial function} and a small step size the $\theta$ method displays the stability property of the underlying equation.  The stability analysis of numerical methods can also be found in \cite{TangBaker97, HB89}.

\end{itemize}

Our main contributions can therefore be highlighted as follows:

\begin{itemize}

  \item A novel approximation method is proposed by the truncation technique. More precisely,
  for a given  type-(B) SFDE, we define a corresponding truncated SFDE with finite time delay $k$.
  The general approximation theory is established by showing that
  the solution of the truncated SFDE approximates the solution of the given  type-(B) SFDE in the $q$th moment provided $k$ is sufficiently large.

   \item Various approximation principles, including the exponential approximation rate, are given for a number of important  type-(B) SFDEs.

  \item Numerical solutions of the truncated SFDE are shown to be close to the  solution of the given  type-(B) SFDE for sufficiently large $k$ and small numerical step size.

       \item For  the global  Lipschitz case  the convergence error  between  the EM numerical solution of type-(A) SFDE and the exact solution of type-(B) SFDE  is given.
\end{itemize}

The rest of the paper is organized as follows: Section 2 gives some necessary notions and assumptions which ensure the well-posedness of the solutions of type-(B) SFDEs. In section 3,
the corresponding truncated SFDE with finite time delay $k$ is defined for a given  type-(B) SFDE, while both asymptotic approximation Theorem \ref{T3.3} and exponential approximation Theorem \ref{T3.5} are established. Section 4 discusses a number of important type-(B) SFDEs  to which the approximation theory established in Section 3 is applicable.
 Section 5 shows that the numerical solutions of the truncated SFDE are  close to the  solution of the given  type-(B) SFDE for sufficiently large $k$ and small numerical step size.
 Two examples with computer simulations are discussed to illustrate the theory.

\section{Preliminaries}

Throughout this paper, unless otherwise specified,
we let $\RR^n$ be the $n$-dimensional Euclidean space and $\BB(\RR^n)$ denote the family of all Borel measurable sets in $\RR^n$.  Let $\RR_+= [0,\8)$ and $\RR_-=(-\8,0]$.
 If $x\in \RR^n$, then $|x|$ is its Euclidean norm.
If $A$ is
a vector or matrix, its transpose is denoted by $A^T$.
If $A$ is a matrix, we let $|A| = \sqrt{\trace(A^TA)}$ be its trace norm and $\|A\| = \max\{|Ax|: |x|=1\}$ be the operator norm. If $A$ is a symmetric matrix ($A=A^T$), denote by $\lambda_{\min}(A)$ and $\lambda_{\max}(A)$
its smallest and largest eigenvalues, respectively. By $A > 0$ and $A\ge 0$,
we mean $A$ is positive and non-negative definite, respectively.
If both $a, b$ are real numbers, then $a \we b = \min\{a, b\}$ and
$a\ve b=\max\{a,b\}$.  Let $\NN_+$ denote the set of nonnegative integers.

We let $(\W ,
{\mathcal F}, \{{\mathcal F}_t\}_{t\ge 0}, \PP)$ be a complete probability
space with a filtration  $\{{\mathcal F}_t\}_{t\ge 0}$ satisfying the
usual conditions (i.e. it is right continuous and increasing while ${\mathcal F}_0$
contains all $\PP$-null sets). For a subset $\bar\W$ of $\W$, $\II_{\bar\W}$ denotes its indicator function. Let $B(t)  = (B_1(t),\cdots,B_m(t))^T$
be an $m$-dimensional Brownian motion defined on the probability space.
Let $\theta(t)$, $t\ge 0$, be a right-continuous irreducible Markov chain on the probability space
taking values in a finite state space $\SM=\{1, 2, \cdots, N\}$ with generator
$\G=(\g_{ij})_{N\K N}$ given by
$$
\PP\{\theta(t+\D)=j | \theta(t)=i\}=
\begin{cases}
 \g_{ij}\D + o(\D) & \hbox{if }
i\not= j, \\
        1+\g_{ii}\D + o(\D) & \hbox{if }  i=j,
        \end{cases}
$$
where $\D>0$.  Here $\g_{ij} \ge 0$ is the transition rate from $i$ to $j$ if
$i\not=j$ while
$\g_{ii} = -\sum_{j\not= i} \g_{ij}$.
 We assume that the Markov chain $\theta(\cdot)$ is independent of the Brownian motion
$B(\cdot)$.

Denote by $C(\RR_-; \RR^n)$ the family of continuous functions
$\f: \RR_- \to \RR^n$.  Other families e.g. $C(\RR^n\K\RR; \RR_+)$ can be defined obviously.
Fix a positive number $r$ and define the phase space $\CC_r$ with the fading memory  by
\begin{equation}\label{Cr}
 \CC_r=\left\{ \f \in C(\RR_-;\RR^n): \sup_{ -\infty <u \leq 0}e^{r u}|\f(u)| < \infty \right\}
\end{equation}
with its norm $\|\f\|_r = \sup_{-\infty < u \le 0} e^{r u}|\f(u)|$.  It is well known that $\CC_r$ under the norm
$\|\cdot\|_r$ forms a Banach space (see \cite{HNM91}  for more details on this phase space).

Moreover,  denote by $\CP_0$ the family of  probability measures $\mu$ on
$\RR_-$, while for  each $b >0$, define
$$
\CP_b = \{ \mu\in \CP_0:  \int_{-\8}^0 e^{-b u}\mu(\d u) < \8\}.
$$
Furthermore, set $\mu^{(b)} := \int_{-\8}^0 e^{-b u} \mu(du)$
 for each $\mu\in \CP_b$. Please note that $\mu^{(b)}$ is a positive number but
 $\mu(\cdot)$ is a measure.  Clearly, $\CP_{b_1} \subset \CP_b\subset \CP_0$
 if $b_1 >b >0$.  Moreover, if $\mu\in \CP_{b_1}$, then $\mu^{(b)}$
 is a strictly increasing and continuous function of $b$ in $ [0,b_1]$.

Consider a hybrid  stochastic functional differential equation (SFDE) with infinite delay of the form
\begin{align}\label{sfde}
dx(t) = f(x_t,\theta(t),t)\d t + g(x_t,\theta(t),t)\d B(t)
\end{align}
on $t\ge 0$ with the initial data
\begin{align}\label{id}
x_0=\xi\in \CC_r \hbox{ and  } \theta(0)=i_0\in\SM.
\end{align}
Here
$f: \CC_r\K \SM\K \RR_+ \to \RR^n$
and $g: \CC_r\K \SM\K \RR_+ \to \RR^{n\K m}$ are Borel measurable and, moreover, $x(t)$ is an $\RR^n$-valued stochastic process on $t\in (-\infty, \infty)$ while
$x_t=\{x(t+u): u\in \RR_-\}$ is a $\CC_r$-valued stochastic process on $t\ge 0$.
We impose the local Lipschitz condition on the coefficients
$f$ and $g$.

\begin{assp}\label{A2.1}
For each number $h >0$, there is a positive constant $\bar K_h$ such that
\begin{align}\label{A2.1a}
 |f(\f,i,t)-f(\p,i,t)|\ve  |g(\f,i,t)-g(\p,i,t)| \le \bar K_h\|\f-\p\|_r
\end{align}
for those $\f,\p\in \CC_r$ with $\|\f\|_r\ve \|\p\|_r\le h$
and all  $(i,t)\in \SM\K \RR_+$.  Moreover,
$$
\sup_{(i,t)\in\SM\K \RR_+} ( |f(0,i,t)|\ve |g(0,i,t)| ) <\8.
$$
\end{assp}

We observe that Assumption \ref{A2.1} only guarantees the existence of
the unique maximal local solution $x(t)$ of the SFDE (\ref{sfde}) on $t\in (-\8, \s_e)$, where $\s_e$ is known as the explosion time
(see, e.g., \cite{M94,WWMZ20,WW07,WYM17}).  To have a global solution (namely, $\s_e=\8$ a.s.), we need some additional conditions.  The classical condition is the linear growth condition (see, e.g.,\cite{KN86,M91,Moh84}).
However, we will impose a much more general Khasminskii-type  condition
(see, e.g., \cite{MR05}).
For this purpose, we need a couple of new notations.
Let $C(\RR^n; \RR_+)$ denote the
family of all continuous functions from $\RR^n$ to $\RR_+$.
Denote by $C^{2,1}(\RR^n\K \SM \K \RR_+; \RR_+)$ the family of all
continuous non-negative functions $V(x,i,t)$ defined on $\RR^n\K \SM
\K \RR_+$ such that for each $i\in \SM$, they are continuously twice differentiable in
$x$ and once in $t$.  Given a function $V\in C^{2,1}(\RR^n\K \SM \K \RR_+; \RR_+)$, we
define the functional $LV: \CC_r\K\SM\K \RR_+ \to \RR$ by
\begin{align*}
& LV(\f,i,t)  = V_t(\f(0),i,t) +
V_x(\f(0),i,t) f(\f,i,t)  \\
& + \frac{1}{2} \trace\big( g(\f,i,t)^T
V_{xx}(\f(0),i,t)g(\f,i,t) \big)
+ \sum_{j=1}^N \g_{ij} V(\f(0),j,t),
\end{align*}
where
$$
  V_t(x,i,t) =   \frac{\M V(x,i,t)}{ \M t}, \quad
 V_x(x,i,t) =   \Big( \frac{\M V(x,i,t)}{\M x_1}, \cdots,
 \frac{\M V(x,i,t)}{\M x_n} \Big),
$$
$$
V_{xx}(x,i,t) =  \Big( \frac{\M^2 V(x,i,t)}{\M x_i\M x_j}\Big)_{n\K n}.
$$
Let us emphasize that $LV$ is defined on
$\CC_r\K \SM\K \RR_+ $
while $V$  on $\RR^n\K \SM \K \RR_+$.  The definition of $LV$ is purely based
on the following generalised It\^o formula (see, e.g., \cite[Theorem 1.45 on p.48]{MY06})
\begin{equation} \label{ItoF}
\d V(x(t),\theta(t),t) = LV(x_t,\theta(t),t)\d t + \d M(t),
\end{equation}
where $M(t)$ is a local martingale with $M(0)=0$ (whose form is of no use in
this paper).   Moreover, for each positive number $b$, define
\begin{align*}
  \CW_b= \bigg\{W \in C(\RR^n\K\RR; \RR_+) : \sup_{(x,u)\in \RR^n\K \RR_-} \frac{W(x,u)}{1+|x|^b} < \8
    \bigg\}.
\end{align*}
Please note that although $W$ is defined on $\RR^n\K\RR$, we only need $W(x,u)/(1+|x|^b)$ to be bounded on $(x,u)\in \RR^n\K \RR_-$.
It is also easy to see that if $W\in \CW_b$, then
\begin{align} \label{Wb}
\sup_{-\8 < u\le 0} e^{rbu} W(\f(u),u) < \8, \ \ \ \forall \f\in \CC_r.
\end{align}
Moreover, denote by $\CK_\8$ the family of non-decreasing functions $\beta: \RR_+\to\RR_+$ such that
$\beta(v) \to \8$ as $v\to\8$.   We can now form the generalised Khasminskii-type condition.

\begin{assp} \label{A2.2}  There are positive constants $b_1, b_2, K$, functions
$V\in C^{2,1}(\RR^n\K \SM \K \RR_+; \RR_+)$, $W_1\in \CW_{b_1}$, $W_2\in \CW_{b_2}$,
$\beta \in \CK_\8$,
and probability measures $\mu_1\in \CP_{rb_1}$,  $\mu_2 \in \CP_{rb_2}$,
such that
\begin{align} \label{A2.2c}
 \beta(|x|) \le  \inf_{0\le t < \8} W_1(x,t),  \ \ \forall x\in \RR^n,
\end{align}
\begin{align} \label{A2.2a}
W_1(x,t) \le V(x,i,t),\ \  \forall (x,i,t)\in \RR^n\K\SM\K\RR_+,
\end{align}
and
\begin{align} \label{A2.2b}
LV(\f,i,t)&  \le K + K \int_{-\8}^0 W_1(\f(u),t+u)\mu_1(\d u) \nonumber \\
& -W_2(\f(0),t) + \int_{-\8}^0 W_2(\f(u),t+u)\mu_2(\d u)
\end{align}
for all $(\f,i,t) \in\CC_r\K\SM\K\RR_+$.
\end{assp}

The following theorem forms the foundation for this paper.

\begin{theorem} \label{T2.3}
  Under Assumptions \ref{A2.1} and \ref{A2.2}, the SFDE (\ref{sfde}) with the initial data (\ref{id})  has a unique solution $x(t)$ on $t\in (-\8,\8)$, which has the property
  that
  \begin{align}\label{T2.3a}
    \sup_{0\le t\le T} \E W_1(x(t),t) \le C_T, \ \ \forall T >0,
  \end{align}
  where $C_T$ is a positive constant dependent on $T$.
\end{theorem}

\noindent
{\it Proof}.  To show the unique maximal local solution  $x(t)$ on $t\in (-\8,\T_e)$ is global, we need to show $\T_e=\8$ a.s.  For each sufficiently large
number  $h \ge \|\xi\|_r$ with $\beta(h)>0$, define the stopping time
$$
\T_h = \T_e\we \inf\{t\in [0,\T_e): |x(t)|\ge h\},
$$
where throughout this paper we set $\inf\emptyset =\8$ (as usual $\emptyset$ stands for the empty set).  Obviously, $\T_h$ is increasing in $h$ and $\T_h\le \T_e$ a.s.  It is therefore sufficient if we could show $\lim_{h\to\8} \T_h=\8$ a.s.  Fix $T > 0$ arbitrarily.    By the generalised It\^o formula (see, e.g., \cite{MY06})
and Assumption \ref{A2.2}, we can easily obtain that
\begin{align}\label{T2.3b}
&  \E W_1(x(t\we\T_h), t\we\T_h)   \le K_1 +
K \E  \int_0^{t\we\T_h} \int_{-\8}^0 W_1(x(s+u), s+u) \mu_1(\d u) \d s
 \nonumber \\
&  - \E  \int_0^{t\we\T_h} W_2(x(s),s))\d s + \E \int_0^{t\we\T_h} \int_{-\8}^0 W_2(x(s+u),s+u)\mu_2(\d u) \d s
\end{align}
for $t\in [0, T]$, where $K_1 = V(\xi(0),i_0,0)+KT$. { Due to $W_1\in\mathcal{W}_{b_1}$, there exists a constant $\hat{K} >0$ such that \begin{equation}\label{w1}W_1(x, u)\leq \hat{K}  (1+|x|^{b_1}),~~~~(x,u)\in \RR^n\times \RR_{-}.\end{equation} This together with the fact $\xi \in\mathcal{C}_r$  implies
\begin{equation}\label{w2}
\sup_{ - \8  < u  \le 0 }e^{ rb_1 u}W_1(\xi(u),u)\leq \hat{K} \sup_{ - \8  < u  \le 0 }e^{ rb_1 u} (1+ |\xi(u) |^{b_1} )\leq \hat{K} (1+\|\xi\|^{b_1}_r)=:K_2. \end{equation}
}
 By the Fubini theorem and property \eqref{w1},  we derive that
 \begin{align}\label{T2.3c}
   & \int_0^{t\we\T_h} \int_{-\8}^0 W_1(x(s+u), s+u) \mu_1(\d u) \d s \nonumber\\
   =&  \int_0^{t\we\T_h}  \Big( \int_{-\8}^{-s}  W_1(x(s+u), s+u) \mu_1(\d u)
   +  \int_{-s}^{0}  W_1(x(s+u), s+u) \mu_1(\d u) \Big) \d s  \nonumber\\
   \le & \Big( \sup_{ - \8  < u  \le -s }e^{ rb_1(s+u)}W_1(\xi(s+u),s+u) \Big)
    \int_{0}^{t\we\T_h}\int_{-\8}^{-s} e^{- rb_1(s+u)}\mu_1(\d u) \d s \nonumber\\
    + & \int_{-(t\we\T_h)}^{0}\int_{-u}^{t\we\T_h}W_1(x(s+u), s+u)\d s \mu_1(\d u) \nonumber\\
     \le & \Big( \sup_{ - \8  < u  \le 0 }e^{ rb_1 u}W_1(\xi(u),u) \Big)
   \Big(  \int_{0}^{t\we\T_h}  e^{- rb_1s} \d s \Big)
    \Big( \int_{-\8}^{0} e^{- rb_1u }\mu_1(\d u) \Big)   \nonumber\\
    + & \int_{-\8}^{0} \int_{0}^{t\we\T_h}W_1(x(s), s)\d s  \mu_1(\d u) \nonumber\\
     \le  & (K_2 /rb_1) \mu_1^{(rb_1)}
     +   \int_{0}^{t\we\T_h}W_1(x(s), s)\d s.
 \end{align}
Similarly,
\begin{align}\label{T2.3d}
     \int_0^{t\we\T_h} \int_{-\8}^0 W_2(x(s+u), s+u) \mu_2(\d u) \d s
     \le   (K_3 /rb_2) \mu_2^{(rb_2)}
     +   \int_{0}^{t\we\T_h}W_2(x(s), s)\d s,
 \end{align}
where $K_3 = \sup_{ - \8  < u  \le 0 }e^{ rb_2 u}W_2(\xi(u),u)$.
Substituting (\ref{T2.3c})
and (\ref{T2.3d}) into (\ref{T2.3b}) yields
\begin{align*}
 \E W_1(x(t\we\T_h), t\we\T_h )  &  \le K_4 + K  \E  \int_0^{t\we\T_h}  W_1(x(s),s)\d s
\\
  &   \le    K_4 + K  \E  \int_0^{t}  \E W_1(x(s\we\T_h), s\we\T_h)\d s,
 \end{align*}
where $K_4= K_1+ (K_2 /rb_1) \mu_1^{(rb_1)} +(K_3 /rb_2) \mu_2^{(rb_2)} $.  The well-known Gronwall inequality shows
\begin{align}\label{T2.3e}
 \E W_1(x(t\we\T_h),  t\we\T_h ) \le C_T,  \ \forall t\in [0,T],
 \end{align}
 where $C_T= K_4 e^{K T}$.
This along with  (\ref{A2.2c}) implies that $C_T \ge  \E \be(|x(T\we\T_h)|)  \ge \beta(h) \PP(\T_h \le T)$, namely
\begin{align}\label{T2.3f}
 \PP(\T_h \le T) \le \frac{C_T } {\beta(h)}.
 \end{align}
Consequently, $\lim_{h\to\8} \PP( \T_h \le T) =0$, namely $\lim_{h\to\8} \T_h> T$ a.s.  Since $T >0$ is arbitrary, we must have $\lim_{h\to\8} \T_h =\8$ a.s.   Letting
 $h\to\8$ in (\ref{T2.3e}) we also get the required assertion (\ref{T2.3a}).
 The proof is therefore complete. $\Box$

Let us highlight that property (\ref{T2.3f}) shows that
$\sup_{0\le t\le T} |x(t)|  \le h$ with probability at least $1-C_T/\beta(h)$ for all sufficiently large $h$. In other words,
the solution remains within the compact ball $\{x\in \RR^n: |x|\le h\}$ for the time interval $[0, T]$
with a large probability for all sufficiently large $h$.  This nice property will play its
important role when we discuss the approximate solutions in the next section.

\section{Approximations by SFDEs with finite delay}

\subsection{Truncated SFDEs}

The key aim of this paper is to approximate the solution of the SFDE (\ref{sfde}) by the solution of an SFDE with finite delay.  In turn, we can
numerically approximate the solution of the SFDE with finite delay, and hence obtain the numerical
 solution of the original SFDE (\ref{sfde}) with infinite delay.

We first need to design the corresponding SFDE with finite delay.
For each positive integer $k$, define the truncation mapping
$\pi_k :\CC_r \to \CC_r$ by {
$$
\pi_k(\f)(u)=
\begin{cases}
  \f(u) & \mbox{if } u \in [-k,0], \\
  \f(-k)  & \mbox{if }   u < -k .
\end{cases}
$$}
Define the truncation functions $f_k:\CC_r \K \SM\K \RR_+ \to \RR^n$
and $g_k: \CC_r\K \SM\K \RR_+ \to \RR^{n\K m}$
by
$$
f_k(\f,i,t) = f(\pi_k(\f),i,t)
\hbox{ and }
g_k(\f,i,t) = g(\pi_k(\f),i,t)
$$
respectively.  We observe that both $f_k$ and $g_k$ depend on
the values of $\f$ on the finite time interval $[-k, 0]$
but not the values of $\f$ on $(-\8, -k)$.  In other words,
$f_k$ and $g_k$ can be regarded as functionals defined on
$C([-k,0); \RR^n)\K \SM\K \RR_+$.
Consider the corresponding truncated SFDE
\begin{align} \label{sfde2}
\d x^k(t) = f_k(x^k_t,\theta(t),t)\d t + g_k(x^k_t,\theta(t),t)\d B(t)
\end{align}
on $t\ge 0$ with the initial data { $x^k_0= \xi $ } and $\theta(0)=i_0$.
This is clearly an SFDE with finite delay.

We observe that the truncation functions $f_k$ and $g_k$
preserve Assumption \ref{A2.1} perfectly.  In fact,
for $\f,\p\in \CC_r$ with $\|\f\|_r\ve \|\p\|_r \le h$
and $(i,t)\in \SM\K \RR_+$,
\begin{align}\label{loclip}
 & |f_k(\f,i,t)-f_k(\p,i,t)|\ve  |g_k(\f,i,t)-g_k(\p,i,t)|  \nn\\
 = & |f(\pi_k(\f),i,t)-f_k(\pi_k(\p),i,t)|\ve  |g(\pi_k(\f),i,t)-g_k(\pi_k(\p),i,t)| \nn\\
 \le & \bar K_h \|\pi_k(\f)-\pi_k(\p)\|_r \le \bar K_h\|\f-\p\|_r.
\end{align}
We now show that they also preserve  Assumption \ref{A2.2}.
In fact, for $V\in C^{2,1}(\RR^n\K \SM \K \RR_+; \RR_+)$,
the generalised It\^o formula shows
$$
\d V(x^k(t),\theta(t),t) = L_kV(x^k_t,\theta(t),t)\d t + \d M_k(t),
$$
where $M_k(t)$ is a local martingale with $M_k(0)=0$ and
$L_kV: \CC_r\K \SM\K \RR_+ \to \RR$ is defined by
\begin{align*}
& L_kV(\f,i,t)  = V_t(\f(0),i,t) +
V_x(\f(0),i,t) f_k(\f,i,t)  \\
& + \frac{1}{2} \trace\big( g_k(\f,i,t)^T
V_{xx}(\f(0),i,t)g_k(\f,i,t) \big)
+ \sum_{j=1}^N \g_{ij} V(\f(0),j,t).
\end{align*}
Under Assumption \ref{A2.2}, we then derive that
\begin{align}\label{lv2}
& L_kV(\f,i,t)  = V_t(\pi_k(\f)(0),i,t) +
V_x(\pi_k(\f)(0),i,t) f(\pi_k(\f),i,t) \nn \\
& + \frac{1}{2} \trace\big( g(\pi_k(\f),i,t)^T
V_{xx}(\pi_k(\f)(0),i,t)g(\pi_k(\f),i,t) \big)
+ \sum_{j=1}^N \g_{ij} V(\pi_k(\f)(0),j,t)\nn \\
& \le K + K \int_{-\8}^0 W_1(\pi_k(\f)(u),t+u)\mu_1(\d u) \nn  \\
& -W_2(\f(0),t) + \int_{-\8}^0 W_2(\pi_k(\f)(u),t+u)\mu_2(\d u).
\end{align}
We can therefore show the following theorem in the same way as Theorem \ref{T2.3} was proved.

\begin{theorem} \label{T3.1}
  Under Assumptions \ref{A2.1} and \ref{A2.2}, for each integer $k\ge 1$ and each $0<T\leq k$, the
  truncated SFDE (\ref{sfde2}) has a unique global solution $x^k(t)$ on $t\in (-\8,T]$.  Moreover,  there is a positive constant $C_T$,
  which is the same as in Theorem \ref{T2.3}, { such that}
{
\begin{align}\label{T3.1a}
\sup_{0\le t\le T} \E W_1(x^k(t),t) \le  C_T, \ \ \forall~ 0<T\leq k,
\end{align} }
  and
  \begin{align}\label{T3.1b}
\PP(\rho_h^k \le T) \le \frac{ C_T} {\beta(h)}
 \end{align}
 for all sufficiently large number $h$,  where $\r_h^{k}= \inf\{t\in[0, \r_e^k): |x^k(t)|\ge h\}.$
 \end{theorem}
\begin{proof}
By virtue of \eqref{loclip} we know that the truncated SFDE \eqref{sfde2} has a  unique solution $x^k(t)$ on $t\in(-\infty, \r_e^k)$, where $\r_{e}^{ k}$ is the explosion time. For any sufficient large number $h$ with $h\ge\|\xi\|_r$ , {$\r_h^k$ is defined as above.}
Obviously, $\r_h^{k}$ is increasing with respect to $h$ and $\r_h^{k}\le \r_e^k\ a.s.$   
 By the generalised It\^o formula 
and  \eqref{lv2}, we   obtain that
\begin{equation}\label{eq0.8}
\begin{aligned}
&\E V(x^k(t\wedge\r_h^{k}),\theta(t\wedge\r_h^{k}), t\wedge\r_h^{k})\le K_1+K\E\int_0^{t\wedge\r_h^{k}}\int_{-\infty}^0W_1(\pi_k(x^k_s)(u), s+u)\mu_1(\d u)\d s\\
&-\E\int_0^{t\wedge\r_h^{k}}W_2(x^k(s),s)\d s+\E\int_0^{t\wedge\r_h^{k}}\int_{-\infty}^0W_2(\pi_k(x^k_s)(u),s+u)\mu_2(\d u)\d s.
\end{aligned}
\end{equation}
for $t\geq 0$, where $K_1={ \E V(\xi(0), i_0, 0)}+KT$.  Recalling the definition of $\pi_k$ one observes
{ $$
\pi_k(x^k_s)(u)=\left\{
\begin{array}{ll}
x^k(s+u),                 &{-k\le u\le 0},\\
x^k(s-k) ,    &{  u <-k}  .
\end{array}
\right.
$$}
Then for any $t\in [0, T]$, $T\leq k$, we derive that
\begin{equation}\label{eq0.9+}
\begin{aligned}
&   \int_0^{t\wedge\r_h^{k}}\int_{-\infty}^0W_1(\pi_k(x^k_s)(u), s+u)\mu_1(\d u)\d s\\
 =&\int_0^{t\wedge\r_h^{k}}\int_{-\infty}^{-s}W_1(\pi_k(x^k_s)(u), s+u)\mu_1(\d u)\d s+{\int_0^{t\wedge\r_h^{k}}\int_{-s}^0W_1(x^k(s+u), s+u)\mu_1(\d u)\d s} \\
 \le& \int_0^{t }\int_{-\infty}^{-s}\Big(\sup_{0\leq s\leq t }\sup_{-\infty<u\le -s}e^{rb_1(s+u)}   W_1(\pi_k(x^k_s)(u), s+u)\Big)e^{-rb_1(s+u)}\mu_1(\d u)\d s\\
&+  \int_0^{t\wedge\r_h^{k}}\int_{- s  }^0  W_1( x^k(s+u), s+u)\mu_1(\d u)\d s,
\end{aligned}
\end{equation}
   It follows from \eqref{w1} and   $\xi \in\mathcal{C}_r$ that
 \begin{align}\label{eq1}
 &   \sup_{0\leq s\leq t }\sup_{-\infty<u\le -s}e^{rb_1(s+u)}  W_1(\pi_k(x^k_s)(u), s+u)\nn\\
&\leq  \hat{K}  \sup_{0\leq s\leq t }\sup_{-\infty<u\le -s}e^{rb_1(s+u)} (1+ |\pi_k(x^k_s)(u)|^{b_1})\nn\\
&\leq    \hat{K}  \sup_{0\leq s\leq t }\Big (1+ \sup_{ -k\le u\le -s}e^{rb_1 (s+u)}\E| x^k (s+u )|^{b_1}\Big)\nn\\
&\leq   \hat{K} +\hat{K} \sup_{-k\leq u\le 0}e^{rb_1 u}  | \xi (u )|^{b_1} \leq K_2,
\end{align}   where $K_2$ is given in \eqref{w2}. Inserting the above inequality into \eqref{eq0.9+} yields
\begin{equation}\label{eq0.9}
\begin{aligned}
& \int_0^{t\wedge\r_h^{k}}\int_{-\infty}^0W_1(\pi_k(x^k_s)(u), s+u)\mu_1(\d u)\d s\\
\le& K_2 \int_0^{t }\int_{-\infty}^{-s}e^{-rb_1(s+u)}\mu_1(\d u)\d s
+ \int_{-t\wedge\r_h^{k}}^0\int_{-u}^{t\wedge\r_h^{k}}W_1(x^k(s+u), s+u)\d s\mu_1(\d u) \\
\le& K_2 \int_0^{t }e^{-rb_1 s }\d s\int_{-\infty}^{0}e^{-rb_1u}\mu_1(\d u)
+  \int_{-\infty}^0\int_{0}^{t\wedge\r_h^{k}}W_1(x^k(s), s)\d s\mu_1(\d u) \\
\le&(K_2/rb_1)\mu_1^{(rb_1)}+ \int_0^{t\wedge\r_h^{k}}W_1(x^k(s), s)\d s.
\end{aligned}
\end{equation}
Similarly,
\begin{equation}\label{eq0.10}
 \int_0^{t\wedge\r_h^{k}}\int_{-\infty}^0W_2(\pi_k(x^k_s)(u), s+u)\mu_2(\d u)\d s\le(K_3 /rb_2)\mu_2^{rb_2}+  \int_0^{t\wedge\r_h^{k}}W_2(x^k(s),s)\d s.
\end{equation}
 Substituting \eqref{eq0.9} and \eqref{eq0.10} into \eqref{eq0.8} gives
\begin{equation}\label{eq2}
\begin{aligned}
\E V(x^k(t\wedge\r_h^{k}),\theta(t\wedge\r_h^{k}), t\wedge\r_h^{k})
&\le K_4+K\int_0^{t}\E W_1(x^k(s\wedge\r_h^{k}), s\wedge\r_h^{k})\d s.
\end{aligned}
\end{equation}
According to the Gronwall inequality and $W_1(x,t)\leq V(x, i,t)$, we have
\begin{equation}\label{eq0.11}
\E W_1(x^k(t\wedge\r_h^{k}),t\wedge\r_h^{k})\le C_{T }, \ \ \ \forall t\in [0, T],~~T\leq k,
\end{equation}
where $C_{T}=K_4e^{KT}$.  By the same way as Theorem \ref{T2.3}, we can get the desired assertions. To avoid the duplication we omit the last proof.
\end{proof}
\subsection{Asymptotic approximations}

Recall that our main aim in this paper is to show
\begin{align} \label{aim}
\lim_{k\to\8} \E|x(t) - x^k(t)|^q = 0, \ \ \forall t > 0,
\end{align}
for $q\ge2$.
It is even more desired if an error estimate can be obtained.
To guarantee the finite $q$th moment of the solution, we slightly strengthen Assumption \ref{A2.2}.

\begin{assp} \label{A3.2}   Assumption \ref{A2.2} holds with $\beta\in \CK_\8$ being defined by
$\beta(u)=u^p$ on $u\in \RR_+$ for some number $p>2$.
\end{assp}

We also need a slightly stronger local Lipschitz condition.

\begin{assp} \label{A3.3}  Let $p$ be the same as in Assumption \ref{A3.2}. Assume that there is a
probability measure  $\mu_3\in {\cal P}_b$ with $b> r$
and a positive constant $K_h$ for each $h >0$ such that
\begin{align} \label{A3.3a}
 |f(\f,i,t)-f(\p,i,t)|\ve  |g(\f,i,t)-g(\p,i,t)|
    \le K_h \int_{-\8}^0 |\f(u)-\p(u)| \mu_3(\d u)
\end{align}
for those $\f,\p\in \CC_r$ with $\|\f\|_r\ve \|\p\|_r \le h$
and all  $(i,t)\in \SM\K \RR_+$.
\end{assp}

Noting that (\ref{A3.3a}) implies
\begin{align} \label{A3.3c}
 |f(\f,i,t)-f(\p,i,t)|\ve  |g(\f,i,t)-g(\p,i,t)|
    \le K_h \mu_3^{(r)}   \|\f -\p \|_r,
\end{align}
we see  this assumption implies Assumption \ref{A2.1}.
Theorems \ref{T2.3} and \ref{T3.1} show that under Assumptions \ref{A3.2} and \ref{A3.3}, the solutions of
equations (\ref{sfde}) and (\ref{sfde2}) satisfy
{
\begin{align}\label{pth}
     \sup_{0\le t\le T} \Big( \E |x(t)|^p \ve  \E |x^k(t)|^p \Big) \le C_T,\ \ \ \forall k\ge T > 0,
  \end{align}
}
  while (\ref{T2.3f}) and (\ref{T3.1b}) become
  \begin{align}\label{stoppings}
 \PP(\T_h \le T) \le \frac{C_T } {h^p}
 \ \hbox{ and } \ \PP(\rho_h^k \le T) \le \frac{ C_T} {h^p}
 \end{align}
  respectively.  The following theorem further shows that $x^k(\cdot)$ converge to $x(\cdot)$ in $L^q$.

\begin{theorem} \label{T3.3}
  Let Assumptions \ref{A3.2} and \ref{A3.3} hold. Then, for each $q\in [2,p)$,
  the solutions of equations (\ref{sfde}) and (\ref{sfde2}) have the property
  that
  \begin{align}\label{T3.3a}
\lim_{k\to\8} \Big( \sup_{0\le t\le T} \E|x(t)-x^k(t)|^q \Big) = 0, \ \ \forall T >0.
  \end{align}
\end{theorem}

\noindent
{\it Proof}. Fix $T >0$ arbitrarily and let ${k \ge T}$. For each $h> \|\xi\|_r$,
let $\T_h$ and $\r_h^{k}$ be the same as defined in the proof of Theorem \ref{T2.3} (noting that we have already proved $\T_e=\8$ a.s.) and in the statement of Theorem \ref{T3.1}, respectively.  Set
$$
\s_h^{k}=\T_h \we\r_h^{k}
\quad\hbox{and} \quad
e^k(t)= x(t)-x^k(t) \ \hbox{ for } { t\in (-\8,T] }.
$$
For any $\de >0$ and $t\in [0,T]$, we can derive by
the Young inequality that
\begin{align} \label{T3.3b}
& \E|e^k(t)|^q   = \E\Big( |e^k(t)|^q \II_{\{\s_h^{k}>T\}} \Big)
+ \E\Big( |e^k(t)|^q \II_{\{\s_h^{k}\le T\} } \Big)
\nonumber \\
 \le &  \E\Big( |e^k(t)|^q \II_{\{\s_h^{k}>T\}} \Big)
+ \frac{q\de}{p} \E|e^k(t)|^p + \frac{p-q}{p\de^{q/(p-q)}} \PP(\s_h^{k}\le T).
\end{align}
By (\ref{pth}) and  (\ref{stoppings}), we have
$$
\E|e^k(t)|^p \le 2^p C_T
$$
and
$$
\PP(\s_h^{k}\le T) \le \PP(\T_h\le T)+\PP(\r_h^{k}\le T) \le \frac{2C_T}{h^p}.
$$
We hence have
\begin{align} \label{T3.3c}
\E|e^k(t)|^q \le \E\Big( |e^k(t )|^q \II_{\{\s_h^{k}>T\}} \Big)
+ \frac{2^pC_T q\de }{p} + \frac{2C_T(p-q)}{p h^p \de^{q/(p-q)}}.
\end{align}
Now, let $\e>0$ be arbitrary. Choosing $\de$ sufficiently small for
$2^p C_T q\de/p \le \e/3$ and then $h$ sufficiently large for
$ 2C_T(p-q)/( p h^p \de^{q/(p-q)} ) \le  \e/3$,
we see from (\ref{T3.3c}) that for this particularly chosen $h$,
\begin{align} \label{T3.3d}
\E|e^k(t)|^q \le \E\Big( |e^k(t )|^q \II_{\{\s_h^{k}>T\}} \Big)
+\frac{2\e}{3}.
\end{align}
If we can show that for all sufficiently large $k$,
\begin{align} \label{T3.3e}
 \sup_{0\le t\le T} \E\Big( |e^k(t )|^q \II_{\{\s_h^{k}>T\}} \Big) \le \frac{ \e}{3},
\end{align}
we then have
$$
\lim_{k\to \8} \Big( \sup_{0\le t\le T} \E|e^k(t)|^q \Big) = 0,
$$
which is the required assertion (\ref{T3.3a}).
In other words, to complete our proof, all we need to do from now on is to show (\ref{T3.3e})
for the particularly chosen $h$.

For $t\in [0,T]$, it follows from (\ref{sfde}) and (\ref{sfde2}) as well as the definition of $f_k$ and $g_k$ that
\begin{align*}
& \E \Big( \sup_{0\le v\le t} |e^k(v\we\s_h^{k})|^q \Big)  \\
  \le & 2^{q-1}  \E \Big( \sup_{0\le v\le t}  \Big| \int_0^{v\we\s_h^{k}} [ f(x_s,\theta(s),s) - f(\pi_k(x^k_s),\theta(s),s) ] \d s \Big|^q \Big)
\\
 +  & 2^{q-1} \E \Big( \sup_{0\le v\le t}  \Big| \int_0^{v\we\s_h^{k}} [ g(x_s,\theta(s),s) - g(\pi_k(x^k_s),\theta(s),s) ] \d B(s)\Big|^q \Big).
\end{align*}
By the H\"older inequality, the Burkholder-David-Gundy inequality
(see, e.g., \cite{M91,M97}) as well as Assumption \ref{A3.3}, it is not difficult to show that{
\begin{align} \label{T3.3f}
  \E \Big( \sup_{0\le v\le t} |e^k(v\we\s_h^{k})|^q \Big)
 \le &  K_5 \E  \int_0^{t\we\s_h^{k}} \Big( \int_{-\8}^0  |x_s( u) - \pi_k(x^k_s)( u)|
 \mu_3(\d u) \Big)^q\d s,
\end{align}
where $K_5= 2^{q-1}K_h^q [T^{q-1}+ (q^{q+1}/2(q-1)^{q-1})^{q/2} T^{(q-2)/2} ]$.
But, for $0 \le s \le t\we\s_h^{k}$ (which is less than $T < k$),
\begin{align}\label{T3.3f2}
   \int_{-\8}^0  { |x_s( u) - \pi_k(x^k_s)( u)|} \mu_3(\d u)
   =  \int_{-s}^0 |e^k(s+u)| \mu_3(\d u) + J ,
\end{align}
where
$$
 J  =
 \int_{-\8}^{-s} |x(s+u) -  { \pi_k(x^k_s)( u)}|
 \mu_3(\d u).
$$
Noting that $x(s+u) = \xi(s+u)$ for $u\le -s$ while {
$$
\pi_k(x^k_s) ( u) =
\begin{cases}
 \xi(s+u) & \mbox{if } -k \le u\le -s, \\
 \xi(s-k)  & \mbox{if }  u< -k.
\end{cases}
$$
We derive that 
\begin{align*}
J  & = \int_{-\infty }^{-k} |\xi(s+u)- \xi(s-k)|  \mu_3(\d u) \\
& \le  \int_{-\infty}^{-k}  \big( |\xi(s+u)| + |\xi(s-k)|  \big) \mu_3(\d u) \\
& \le  \int_{-\infty}^{-k} \big( \|\xi\|_r e^{-r(s+u)} + \|\xi\|_r e^{-r(s-k)} \big)  \mu_3(\d u) \\
 & \le  2 \|\xi\|_r \int_{-\infty}^{-k}  e^{-r(s+u)} \mu_3(\d u)\\
 & \le  2 \|\xi\|_r\int_{-\8}^{-k}  e^{-bu +(b-r)u} \mu_3(\d u) \\
&  \le  2 \|\xi\|_r  e^{-(b-r)k } \int_{-\8}^{-k}  e^{-bu } \mu_3(\d u)
 \le K_6 e^{-(b-r)k},
 \end{align*} }
 where $K_6= 2 \|\xi\|_r  \mu_3^{(b)}$.   Putting this into (\ref{T3.3f2}) we obtain
\begin{align*}
  \int_{-\8}^0 {|x_s( u)-\pi_k(x^k_s)( u)|} \mu_3(\d u)
 \le \int_{-s}^0|e^k(s+u)| \mu_3(\d u) + K_6 e^{-(b-r)k}.
 \end{align*}
Then
\begin{align}\label{T3.3g}
    \Big(\int_{-\8}^0 {|x_s( u)-\pi_k(x^k_s)( u)|} \mu_3(\d u)\Big)^q
 \le 2^{q-1}\int_{-s}^0|e^k(s+u)|^q \mu_3(\d u) + 2^{q-1}K_6^q e^{-q(b-r)k}.
\end{align}
Substituting this into (\ref{T3.3f}) yields
\begin{align} \label{T3.3h}
 & \E \Big( \sup_{0\le v\le t} |e^k(v\we\s_h^{k})|^q \Big)
  \nonumber \\
 \le & 2^{q-1} K_5K_6^q T e^{-q(b-r)k}
+  K_5 \E  \int_0^{t\we\s_h^{k}} \Big( \int_{-s}^{0} |e^k(s+u)|^q
 \mu_3(\d u) \Big) \d s.
\end{align}
But
\begin{align} \label{T3.3h2}
  \int_0^{t\we\s_h^{k}} & \Big(  \int_{-s}^0 |e^k(s+u)|^q \mu_3(\d u)
 \Big) \d s
=   \int_{-(t\we\s_h^{k})}^0 \Big( \int_{-u}^{t\we\s_h^{k}}   |e^k(s+u)|^q \d s  \Big) \mu_3(\d u)
 \nonumber \\
\le & \int_{-(t\we\s_h^{k})}^0 \Big(  \int_{0}^{t\we\s_h^{k}}
|e^k(s)|^q \d s  \Big) \mu_3(\d u)
 \le    \int_{0}^{t\we\s_h^{k}}  |e^k(s)|^q \d s.
\end{align}
It therefore follows from (\ref{T3.3h}) that
\begin{align} \label{T3.3i}
   \E \Big( & \sup_{0\le v\le t} |e^k(v\we\s_h^{k})|^q \Big)
\le    2^{q-1} K_5K_6^{q} T e^{-q(b-r)k}
+  K_5 \E  \int_0^{t\we\s_h^{k}} |e^k(s)|^q \d s
\nonumber \\
\le & 2^{q-1} K_5K_6^{q} T e^{-q(b-r)k}
+  K_5   \int_0^{t}  \E \Big( \sup_{0\le v\le s} |e^k(v\we\s_h^{k})|^q \Big) \d s.
\end{align}
An application of the Gronwall inequality yields
$$
 \E \Big( \sup_{0\le v\le T } |e^k(v\we\s_h^{k})|^q \Big)
 \le 2^{q-1} K_5K_6^{q} T e^{K_5T-q(b-r)k}.
$$
Hence
\begin{align} \label{T3.3j}
    \E \Big( \sup_{0\le v\le T } |e^k(v)|^q \II_{\{\s_h^{k}>T\} } \Big)
 \le 2^{q-1}K_5K_6^{q} T e^{K_5T-q(b-r)k}.
\end{align}
This implies (\ref{T3.3e}) of course.  The proof is therefore complete. $\Box$
}
\subsection{Approximations with exponential convergence order}

The convergence of { $x^k(\cdot)$ to $x(\cdot)$ } in Theorem \ref{T3.3} is described
in the asymptotic way. The proof itself provides us with a way to
estimate the error. Namely, for a given $\e >0$, we can
determine a positive integer $k_0=k_0(\e)$ (by choosing $\de$, $h$ first)
so that
$$
\sup_{0\le t\le T} \E|x(t)-x^k(t)|^q < \e, \ \ \forall k\ge k_0.
$$
On the other hand, we observe that (\ref{T3.3j}) gives an exponential estimate on the error up to the stopping time $\s_h^{k}\we T$.  If we can remove the stopping time there, we will have a
stronger convergence order of { $x^k(\cdot)$ to $x(\cdot)$ }. In this sub-section, we aim to establish  the exponential convergence order described by
$$
\limsup_{k\to\8} \frac{1}{k} \log\Big( \E|x(T)-x^k(T)|^q \Big) < 0, \ \ \forall T >0.
$$
We need a couple of new notations. Given a function $\bV \in C^{2,1}(\RR^n\K \SM \K \RR_+; \RR_+)$, we
define the functional $\CL\bV: \CC_r \K \CC_r \K \SM\K \RR_+ \to \RR$ by
\begin{align*}
& \CL\bV(\f,\p,i,t)  = \bV_t(\f(0)-\p(0),i,t) +
\bV_x(\f(0)-\p(0),i,t) [f(\f,i,t) - f(\p,i,t)] \\
& + \frac{1}{2} \trace\big( [g(\f,i,t) - g(\p,i,t)]^T
\bV_{xx}(\f(0)-\p(0),i,t)[g(\f,i,t) - g(\p,i,t)] \big)
\\
&  +
\sum_{j=1}^N \g_{ij} \bV(\f(0)-\p(0),j,t).
\end{align*}
Let us emphasize that $\CL\bV$ is defined on $\CC_r\K \CC_r
\K \SM\K \RR_+ $
while $\bV$  on $\RR^n\K \SM \K \RR_+$.  For $\be>0$, let ${\cal U}_{0,\be}$  denote the family of
continuous functions $U: \RR^n\K\RR^n \to \RR_+$ such that $U(x,y)=0$ whenever $x=y\in \RR^n$ while
$$
\sup_{x,y\in \RR^n} \frac{U(x,y)}{1+|x|^\be+|y|^\be} < \8.
$$
It is easy to see
\begin{align} \label{Ub}
\sup_{-\8 < u\le 0} e^{r\be u} U(\f(u),\p(u)) < \8, \ \ \ \forall \f, \p \in \CC_r.
\end{align}
For example, if $U(x,y)= |x-y|(1+|x|+|y|)$ for $(x,y)\in\RR^n\K\RR^n$, then $U \in {\cal U}_{0,2}$.

\begin{assp} \label{A3.4}
There are positive numbers $\be, \bq, \bar K, b_4, b_5$, functions $\bV\in C^{2,1}(\RR^n\K \SM\K \RR_+; \RR_+)$,
$U \in {\cal U}_{0,\be}$,  as well as two probability measures
 $\mu_4 \in \CP_{b_4}$,  $\mu_5\in \CP_{b_5}$,  such that $b_4> r\bq$, $b_5> r\be$,
 \begin{equation} \label{A3.4a}
  \bV(0,i,t) = 0  , \ \ \forall (i,t)\in \SM \K \RR_+,
\end{equation}
\begin{equation} \label{A3.4b}
|x|^{\bq} \le \bV(x,i,t)  , \ \ \forall (x,i,t)\in \RR^n\K \SM \K \RR_+,
\end{equation}
and
 \begin{align} \label{A3.4c}
&  \CL\bV(\f,\p,i,t) \le \bar K \int_{-\8}^0 |\f(u)-\p(u)|^{\bq} \mu_4(\d u)
\nonumber \\
&  -U(\f(0),\p(0))+  \int_{-\8}^0 U(\f(u),\p(u))  \mu_5(\d u)
 \end{align}
for all $(\f,\p,i,t)\in \CC_r\K \CC_r \K \SM\K \RR_+$.
\end{assp}

\begin{theorem} \label{T3.5}
  Let Assumptions \ref{A2.1}, \ref{A2.2} and  \ref{A3.4} hold.  Set $\lbd  = (b_4 - r\bq) \we { (b_5- r\be)}$. Then
  the solutions of equations (\ref{sfde}) and (\ref{sfde2}) have the properties
  that, for all $T >0$,
  \begin{align}\label{T3.5a}
  \limsup_{k\to\8} \frac{1}{k}  \log\big(  \E|x(T)-x^k(T)|^{\bq} \big) \le -\lbd,
  \end{align}
  and
   \begin{align}\label{T3.5a2}
  \limsup_{k\to\8} \frac{1}{k}  \log (|x(T)-x^k(T)|)  \le -\frac{\lbd}{\bq}  \ \ \hbox{ a.s. }
  \end{align}
\end{theorem}

\noindent
{\it Proof}.  Fix any $T>0 $ and integer $k >T$.  Let $h > \|\xi\|_r$.   Let $e^k(t)$ and $\s_h^{k}$ be the
same as defined in the proof of Theorem \ref{T3.3}.
Obviously, $\s_h^{k}\to\8$ almost surely as $h\to\8$.
By the generalised It\^o formula as well as the definitions of $f_k, g_k, \pi_k$, it is straightforward to verify that, for $t\in [0,T]$,
$$
\E \bV(e^k(t\we\s_h^{k}), \theta(t\we\s_h^{k}), t\we\s_h^{k})
= \E \int_0^{t\we\s_h^{k}} \CL \bV(x_s,\pi_k(x^k_s),\theta(s), s) \d s.
$$
By Assumption \ref{A3.4}, we then have
\begin{align} \label{T3.5b}
    \E|e^k(t\we\s_h^{k})|^{\bq}
 &  \le \bar K \E \int_0^{t\we\s_h^{k}} \Big(
   \int_{-\8}^0 |{ x_s( u)-\pi_k(x^k_s)( u) }|^{\bq} \mu_4(\d u) \Big) \d s
   \nonumber \\
  &  -   \E \int_0^{t\we\s_h^{k}} U(x(s),  { x^k (s)}) \d s \nonumber \\
  &  + \E \int_0^{t\we\s_h^{k}} \Big(
   \int_{-\8}^0 U({ x_s( u),  \pi_k(x^k_s)( u)} ) \mu_5(\d u) \Big) \d s.
\end{align}
In the same way as (\ref{T3.3g}) was proved, we can show that
\begin{align} \label{T3.5b2}
 \int_{-\8}^0 {|x_s( u)-\pi_k(x^k_s)( u) |}^{\bq} \mu_4(\d u)
   \le \int_{-s}^0 |e^k(s+u) |^{\bq} \mu_4(\d u) + K_7 e^{-(b_4-r\bq)k}
\end{align}
and
\begin{align} \label{T3.5b3}
& \int_{-\8}^0 U({ x_s( u), \pi_k(x^k_s)( u)} ) \mu_5(\d u) \nonumber \\
   \le & \int_{-s}^0 U(  x(s+u),  x^k(s+u) ) \mu_5(\d u)+ K_8 e^{-(b_5-r\be)k},
\end{align}
where $K_7$ 
and $K_8$ are both positive constants independent of $k$.
Substituting these into (\ref{T3.5b}) and making use of
\begin{align*}
  \int_0^{t\we\s_h^{k}}  \Big(  \int_{-s}^0 |e^k(s+u)|^{\bq} \mu_4(\d u) \Big) \d s
  \le \int_{0}^{t\we\s_h^{k}}  |e^k(s)|^{\bq} \d s
\end{align*}
and
\begin{align*}
  \int_0^{t\we\s_h^{k}}  { \Big(   \int_{-s}^0 U( x(s+u),  x^k(s+u) ) \mu_5(\d u)
 \Big)} \d s
 \le  \int_{0}^{t\we\s_h^{k}}  U( x(s),  x^k(s) )   \d s
\end{align*}
(please see (\ref{T3.3h2})),  we obtain
\begin{align} \label{T3.5e}
\E|e^k(t\we\s_h^{k})|^{\bq}  \le K_9 e^{-\lbd k} + \bar K \int_{0}^t   \E|e^k(s\we\s_h^{k})|^{\bq}   \d s ,
\end{align}
where $\lbd$ has been defined in the statement of the theorem and $K_9=T(\bar K K_7+K_8)$.
An application of the Gronwall inequality implies
$$
    \E|e^k(T \we\s_h^{k})|^{\bq}  \le K_{10} e^{-\lbd k},
$$
where $K_{10}= K_9e^{\bar K T}$. Letting $h\to\8$ yields
\begin{align} \label{T3.5f}
    \E|e^k(T)|^{\bq}  \le K_{10} e^{-\lbd k},
\end{align}
which implies immediately the first required assertion (\ref{T3.5a}).

To show the second assertion (\ref{T3.5a2}), we let $\e\in (0,\lbd)$ be arbitrary.  It follows from (\ref{T3.5f}) that
$$
\PP\{ |e^k(T)|^{\bq} > e^{-(\lbd-\e)k} \}
\le \frac{ \E |e^k(T)|^{\bq} }{ e^{-(\lbd-\e)k} }
\le K_{10}  e^{-\e k}, \ \ \forall k > T.
$$
 By the well-known Borel-Cantelli lemma (see, e.g., \cite[Lemma 2.4 on page 7]{M97}), we can find a subset $\W_0\subset \W$ with $\PP(\W_0)=1$ so that
 for each $\w\in \W_0$, there is an integer $k_0(\w)$ such that
 $$
 |e^k(T,\w)|^{\bq} \le e^{-(\lbd-\e)k}, \ \ \forall k\ge k_0(\w).
 $$
 This yields
 $$
 \lim_{k\to\8} \frac{1}{k} \log(|e^k(T,\w)|) \le -\frac{(\lbd-\e)}{\bq},
 \ \ \forall \w\in\W_0.
 $$
 This further implies the another assertion (\ref{T3.5a2}) as
 $\e$ is arbitrary while $\PP(\W_0)=1$.
The proof is hence complete. $\Box$

\begin{rmk}\label{R3.6}
We observe that both Assumptions \ref{A2.1} and \ref{A2.2}
were not used explicitly in the proof of Theorem \ref{T3.5}.
But they were in fact used to guarantee both SFDEs (\ref{sfde}) and (\ref{sfde2}) have their own unique solution.  In other words,
Theorem \ref{T3.5} holds if  both Assumptions \ref{A2.1} and \ref{A2.2} are replaced by the condition that both SFDEs (\ref{sfde}) and (\ref{sfde2})
have their own unique solution.
\end{rmk}

\section{Important classes of SFDEs}

To show the power of the general approximation theory established in the previous section,
 we will study a couple of important classes of SFDEs
and their approximations in this section.

\subsection{Global Lipschitz}

We start with the class of SFDEs under the global Lipschitz condition.

\begin{assp} \label{A4.1}  There are two constants $c_1 >0$, $p >2$ and a
probability measure $\mu_6\in {\cal P}_{rp}$   such that
\begin{align} \label{A4.1a}
 |f(\f,i,t)-f(\p,i,t)|\ve  |g(\f,i,t)-g(\p,i,t)|
    \le c_1 \int_{-\8}^0 |\f(u)-\p(u)| \mu_6(\d u)
\end{align}
for all $\f,\p\in \CC_r$ and $(i,t)\in \SM\K \RR_+$.  Moreover,
\begin{align} \label{A4.1b}
 \sup_{(i,t)\in \SM\K\RR_+} \big( |f(0,i,t)|\ve  |g(0,i,t)| \big)
   < \8.
\end{align}
\end{assp}

This assumption implies Assumption \ref{A3.3} obviously.  It is also easy to verify that Assumption \ref{A3.2} is satisfied with
$V(x,i,t)=|x|^p$, $W_1(x,t) = |x|^p$, $W_2(x,t)=0$ etc.
 To verify Assumption \ref{A3.4}, we let
$\bV(x,i,t)=|x|^{\bq}$ for  $\bq\in [2,p)$.  From now on, we also let $\de_0(\cdot)$ be
 the Dirac measure at 0, which can be regarded as a probability measure
 on $\RR_-$ and it belongs to $\bigcap_ {b\ge 1}\CP_b$.
For $\f,\p\in \CC_r$ and $(i,t)\in \SM\K \RR_+$,
we can then derive that
\begin{align} \label{A4.1c}
& \CL\bV(\f,\p,i,t)  \nonumber \\
\le & \bq |\f(0)-\p(0)|^{\bq-2}
\Big( (\f(0)-\p(0))^T (f(\f,i,t)-f(\p,i,t))
+\frac{\bq-1}{2}  |g(\f,i,t)-g(\p,i,t)|^2 \Big)
\nonumber \\
\le &  0.5\bq(\bq-1) |\f(0)-\p(0)|^{\bq}
+ \bq \big( |f(\f,i,t)-f(\p,i,t)|^{\bq} \ve
 |g(\f,i,t)-g(\p,i,t)|^{\bq} \big) \nonumber \\
 \le &  0.5\bq(\bq-1)  \int_{-\8}^0 |\f(u)-\p(u)|^{\bq} \de_0(\d u)
+  \bq c_1^{\bq} \int_{-\8}^0 |\f(u)-\p(u)|^{\bq}  \mu_6(\d u)
\nonumber \\
\le & [ 0.5\bq(\bq-1) +   \bq c_1^{\bq} ] \int_{-\8}^0 |\f(u)-\p(u)|^{\bq}  ( \de_0(\d u) + \mu_6(\d u))
\nonumber \\
= &  \bq [\bq-1  + 2 c_1^{\bq} ]  \int_{-\8}^0 |\f(u)-\p(u)|^{\bq} \bar\mu_6(\d u),
\end{align}
where $\bar\mu_6(\cdot) = 0.5(\mu_6(\cdot)+\de_0(\cdot)) \in \CP_{rp}$.
This shows that Assumption \ref{A3.4} is satisfied with
$U(x)=0$ etc.  By Theorem \ref{T3.5}, we can therefore conclude that
under Assumption \ref{A4.1}, the solutions of equations (\ref{sfde}) and (\ref{sfde2}) have   properties (\ref{T3.5a}) and (\ref{T3.5a2})  with $\lbd=(p-\bq)r$    for any $\bq \in [2,p)$.

  \subsection{Khasminiskii case}

 A much wider class of SFDEs is covered by the Khasminskii condition than the global Lipschitz condition. We here propose a special Khasminskii-type condition.

  \begin{assp} \label{A4.3}  There are  constants $c_2 >0$, $\bar p > 2$ and a
probability measure $\mu_7\in {\cal P}_{r\bar p}$
such that
\begin{align} \label{A4.3a}
& (\f(0)-\p(0))^T [f(\f,i,t)-f(\p,i,t)] + \frac{\bar p-1}{2}    |g(\f,i,t)-g(\p,i,t)|^2 \nonumber \\
   & \quad \quad \le c_2 \int_{-\8}^0 |\f(u)-\p(u)|^2  \mu_7(\d u)
\end{align}
for all $\f,\p\in \CC_r$
and $(i,t)\in \SM\K \RR_+$.  Moreover, condition (\ref{A4.1b}) holds.
\end{assp}

This assumption does not in general imply the local Lipschtiz continuity of
both $f$ and $g$, we hence need to retain Assumption \ref{A2.1}.  But this assumption does imply Assumptions \ref{A2.2} and \ref{A3.4}.   Let us verify Assumption  \ref{A2.2} first.  Let $V(x,i,t) = |x|^p$
for any $p\in (2, \bar p)$. For
$(\f,i,t)\in \CC_r\K\SM\K\RR_+$, we have
$$
 LV(\f,i,t)
\le  p |\f(0)|^{p-2}
\Big( \f(0)^T f(\f,i,t)
+\frac{p-1}{2}  |g(\f,i,t)|^2 \Big).
$$
Making use of the elementary inequality
$$
|a+b|^2\le \frac{\bar p - 1}{p-1} a^2 + \frac{\bar p-1}{\bar p-p} b^2,
\ \  \forall a, b\ge 0
$$
as well as Assumption \ref{A4.3}, we can get
$$
 LV(\f,i,t) \le  p |\f(0)|^{p-2} \Big(c_3+ c_3|\f(0)| + c_2 \int_{-\8}^0 |\f(u)|^2  \mu_7(\d u) \Big),
$$
where $c_3$ and the following $c_4$, $c_5$   are all positive constants.
By the well-known Young inequality, we can further get
\begin{align} \label{A4.3b}
 LV(\f,i,t) & \le c_4\Big(1+  \int_{-\8}^0  |\f(u)|^p  \de_0(\d u) + \int_{-\8}^0  |\f(u)|^p  \mu_7(\d u) \Big)  \nonumber \\
 & \le 2c_4 \Big( 1 +   \int_{-\8}^0 |\f(u)|^p \bar\mu_7(\d u) \Big),
\end{align}
where $\bar\mu_7(\cdot) = 0.5(\de_0(\cdot)+\mu_7(\cdot)) \in \CP_{r\bar p}$.  This shows that Assumption \ref{A2.2} is satisfied with $W_1(x,t)=|x|^p$, $W_2(x,t)=0$, etc.  To verify Assumption \ref{A3.4}, we let
$\bV(x,i,t)=|x|^{\bq}$ for $\bq \in [2, \bar p)$.
For $\f,\p\in \CC_r$ and $(i,t)\in \SM\K \RR_+$,
it then follows from the first inequality in (\ref{A4.1c}) and Assumption
\ref{A4.3} that
\begin{align} \label{A4.3c}
& \CL \bV(\f,\p,i,t)  \nonumber \\
\le & \bq c_2 |\f(0)-\p(0)|^{\bq-2}  \int_{-\8}^0 |\f(u)-\p(u)|^2 \mu_7(\d u)
\nonumber \\
 \le &  c_5 \Big( \int_{-\8}^0  |\f(u)-\p(u)|^{\bq}   \de_0(\d u)
 + \int_{-\8}^0  |\f(u)-\p(u)|^{\bq}   \mu_7(\d u) \Big) \nonumber \\
  = &  2 c_5 \int_{-\8}^0  |\f(u)-\p(u)|^{\bq}   \bar\mu_7(\d u).
\end{align}
This shows that Assumption \ref{A3.4} is satisfied with $U(x)=0$
 etc.  By Theorem \ref{T3.5}, we can conclude
  under Assumptions \ref{A2.1} and \ref{A4.3}, the assertions of Theorem \ref{T3.5} hold with $\lbd=
  (\bar p-\bq)r$ for any $\bq \in [2,\bar p)$.

  \subsection{Highly nonlinear SFDEs} \label{SecHN}

  In this subsection we will consider a class of highly nonlinear SFDEs
  in the form
  \begin{align}\label{sfde3}
   \d x(t) = F(x(t),\psi_1(x_t),\theta(t),t) \d t + G(x(t),\psi_2(x_t),\theta(t),t) \d B(t) 
  \end{align}
  on $t\ge 0$ with the initial data (\ref{id}).  Here $F: \RR^n\K\RR^n\K\SM\K\RR_+\to\RR^n$ and  $G: \RR^n\K\RR^n\K\SM\K\RR_+\to\RR^{n\K m}$ are Borel measurable while
  $\psi_1,\psi_2: \CC_r\to \RR^n$ are defined by
  \begin{align}\label{psi12}
   \psi_1(\f)=\int_{-\8}^0 \f(u) \mu_8(\d u)
 \   \hbox{ and }  \
   \psi_2(\f)=\int_{-\8}^0 \f(u) \mu_9(\d u)
  \end{align}
  with $\mu_8, \mu_9\in \CP_0$ (which will be strengthened later).    Equation (\ref{sfde3})
  becomes our underlying SFDE (\ref{sfde}) if we define
$f: \CC_r\K \SM\K \RR_+ \to \RR^n$
and $g: \CC_r\K \SM\K \RR_+ \to \RR^{n\K m}$ by
\begin{align}\label{A4.6b}
f(\f,i,t)= F(\f(0),\psi_1(\f),i,t)
\ \hbox{ and } \
g(\f,i,t)= G(\f(0),\psi_2(\f),i,t)
\end{align}
respectively.  Moreover, the corresponding truncated SFDE (\ref{sfde2}) becomes
{  \begin{align}\label{sfde4}
   \d x^k (t) &= f(\pi_k(x^k_t),\theta(t),t) \d t + g(\pi_k(x^k_t),\theta(t),t) \d B(t)
   \nonumber \\
   & = F(x^k(t),\psi_1(\pi_k(x^k_t)),\theta(t),t) \d t + G(x^k(t),\psi_2(\pi_k(x^k_t)),\theta(t),t) \d B(t)
  \end{align} }
 on $t\ge 0$ with initial data $x^k_0=\pi_k(\xi)$ and $\theta(0)=i_0$.

  To apply our theory established in the previous sections,
  we impose the local Lipschitz condition on $F$ and $G$.

  \begin{assp}\label{A4.5}
For each $h >0$, there is a positive constant $\tl K_h$ such that
\begin{align}\label{A4.5a}
 |F(x,y,i,t)-F(\bar x,\bar y,i,t)|\ve  |G(x,y,i,t)-G(\bar x, \bar y,i,t)| \le \tl K_h(|x-\bar x|+|y-\bar y|)
\end{align}
for all $x, y, \bar x, \bar y \in \RR^n$ with $|x|\ve |y|\ve |\bar x|\ve |\bar y| \le h$
and $(i,t)\in \SM\K \RR_+$.
\end{assp}

We also impose a generalised Khasminskii-type condition (see, e.g., \cite{MR05}).

\begin{assp}\label{A4.6}
There are nonnegative constants $c_6$ - $c_9$, $p$, $\bar p$  such that
$c_7\ge c_8+c_9$, $p> 2$, $\bar p >0$ and
\begin{align}\label{A4.6a}
&  x^T F(x,y,i,t)+\frac{p-1}{2} |G(x,z,i,t)|^2 \nonumber\\
 \le &  c_6(1+|x|^2+|y|^2+|z|^2) - c_7|x|^{2+\bar p} +c_8|y|^{2+\bar p} + c_9|z|^{2+\bar p}
\end{align}
for all $(x,y,z,i,t)\in \RR^n\K\RR^n\K\RR^n\K\SM\K\RR_+$. Moreover, $\mu_8, \mu_9\in \CP_{r(p+\bar p)}$.
\end{assp}

We first verify Assumption  \ref{A3.2} under Assumption \ref{A4.6}.  Let $V(x,i,t)=|x|^p$.  Then
$$
 LV(\f,i,t)
\le  p |\f(0)|^{p-2}
\Big( \f(0)^T F(\f(0),\psi_1(\f),i,t)
+\frac{p-1}{2}  |G(\f(0),\psi_2(\f),i,t)|^2 \Big).
$$
for $(\f,i,t)\in \CC_r\K\SM\K\RR_+$.  By Assumption \ref{A4.6},
\begin{align*}
 LV(\f,i,t)
\le &  p |\f(0)|^{p-2}
\Big( c_6(1+|\f(0)|^2+|\psi_1(\f)|^2 + |\psi_2(\f)|^2 )  \\
 & \quad - c_7|\f(0)|^{2+\bar p} +c_8|\psi_1(\f)|^{2+ \bar p} + c_9|\psi_2(\f)|^{2+ \bar p}   \Big).
\end{align*}
By the Young inequality, we can then easily show that
\begin{align*}
 LV(\f,i,t)
\le &
2p c_6(1+|\f(0)|^p+|\psi_1(\f)|^p +|\psi_2(\f)|^p )  \\
  - & \bar c_7|\f(0)|^{p+\bar p} +\bar c_8|\psi_1(\f)|^{p+\bar p} + \bar c_9|\psi_2(\f)|^{p+\bar p},
\end{align*}
where {$\bar c_7=pc_7-p(c_8+c_9)(p-2)/(p+\bar p)$},
$\bar c_8 = pc_8(2+\bar p)/(p+\bar p)$ and
$\bar c_9 = pc_9(2+\bar p)/(p+\bar p)$ so $\bar c_7 \ge \bar c_8+\bar c_9$. Using $|\psi_1(\f)|^p\le \int_{-\infty}^0 |\f(u)|^p\mu_8(\d u)$, we further get
\begin{align*}
 LV(\f,i,t)
\le &
{4p c_6} \Big( 1+ \int_{-\8}^0 |\f(u)|^p  \de_0(\d u) +\int_{-\8}^0 |\f(u)|^p  \mu_8(\d u)+\int_{-\8}^0 |\f(u)|^p  \mu_9(\d u) \Big)  \\
  - &   \bar c_7|\f(0)|^{p+\bar p} +\bar c_8 \int_{-\8}^0 |\f(u)|^{p+\bar p} \mu_8(\d u)
    + \bar c_9\int_{-\8}^0 |\f(u)|^{p+\bar p} \mu_9(\d u)  \\
   \le &
{4p c_6}\Big( 1+ 3 \int_{-\8}^0 |\f(u)|^p  \mu_{10}(\d u) \Big)  \\
  - &   \bar c_7|\f(0)|^{p+\bar p} +  \bar c_7  \int_{-\8}^0
  |\f(u)|^{p+\bar p}
  \mu_{11}(\d u),
\end{align*}
where $\mu_{10}(\cdot) = [\de_0(\cdot)+\mu_8(\cdot)+\mu_9(\cdot)]/3\in \CP_{r(p+\bar p)}$ and $\mu_{11}(\cdot) = [\bar c_8 \mu_8(\cdot)+\bar c_9 \mu_9(\cdot)]/(\bar c_8 +\bar c_9) \in \CP_{r(p+\bar p)}$.
We therefore see that Assumption \ref{A3.2} is satisfied with
$W_1(x,t)=|x|^p$, $W_2(x,t)=\bar c_7|x|^{p+\bar p}$ etc.

Let us now verify Assumption  \ref{A3.3}.
Let $h>0$ and set $\bar h= h(\mu_8^{(r)}\ve \mu_9^{(r)})$.  For $\f,\p\in \CC_r$ with { $\|\f\|_r\ve \|\p\|_r \le h$}
and $(i,t)\in \SM\K \RR_+$, we derive from Assumption \ref{A4.5} that
\begin{align*}
 & |F(\f(0),\psi_1(\f),i,t)-F(\p(0),\psi_1(\p),i,t)|\ve
 |G(\f(0),\psi_2(\f),i,t)-G(\p(0),\psi_2(\p),i,t)| \\
    \le & \tl K_h \big( \int_{-\8}^0 |\f(u)-\p(u)| \de_0(\d u) + \int_{-\infty}^0 |\f(u)-\p(u)| \mu_8(\d u)+ \int_{-\infty}^0 |\f(u)-\p(u)| \mu_9(\d u) \big) \\
    =& 3 \tl K_ h  \int_{-\8}^0 |\f(u)-\p(u)| \mu_{10}(\d u),
\end{align*}
where $\mu_{10}$ has been defined above.
 We can therefore conclude from Theorem \ref{T3.3} that under Assumptions \ref{A4.5} and \ref{A4.6},
  the solutions of equations (\ref{sfde3}) and (\ref{sfde4}) have   property (\ref{T3.3a})
for each $q\in [2,p)$.

\section{Numerical methods}

The theory established in the previous sections enables us to obtain
numerical approximate solutions to SFDEs with infinite delay.
More precisely, in order to numerically approximate the solution of the SFDE (\ref{sfde}), we can now  obtain the numerical solution of the
corresponding truncated SFDE (\ref{sfde2}) for a sufficiently large $k$.

As pointed out before, the  truncated SFDE (\ref{sfde2}) is an SFDE with finite delay.  Although numerical methods for the SFDEs with finite delay have been studied by many authors
(see, e.g., \cite{GMY18,WM08,WMK11,WMS10,ZSL18}), the existing results can not be applied directly to the SFDE (\ref{sfde2}) due to its special truncated feature.
Fortunately,  numerical methods can be modified to obtain the numerical solutions of
the truncated SFDE (\ref{sfde2}).  To demonstrate the idea, we will concentrate on obtaining   the numerical solutions of the truncated  SFDE (\ref{sfde4}) and hence the numerical  solutions of the SFDE (\ref{sfde3}).

{
\subsection{Lipschitz case}\label{SecRate}

How the existing numerical analysis can be modified to obtain the numerical solutions of the truncated SFDE (\ref{sfde4}) is best illustrated in the globally Lipschitz case.

\begin{assp}\label{A5.4}
There exists a constant $L_1>0$ such that
$$
|F(x, y, i, t)-F(\bar x, \bar y, i, t)|\vee|G(x, y, i, t)-G(\bar x, \bar y, i, t)| \le L_1(|x-\bar x|+ |y-\bar y|)
$$
for all $x, y, \bar x, \bar y\in \RR^n$, $t\in \RR_+$ and $i\in\mathbb{S}$. Moreover,
$$
\sup_{(i,t)\in\SM\K \RR_+} ( |F(0,0, i,t)|\ve |G(0,0, i,t)| ) <\8.
$$
\end{assp}

This assumption of course guarantees that  the truncated SFDE \eqref{sfde4} has a unique global  solution $x^k(t)$.
Let us now fix a sufficiently large $k$ and apply the Euler-Maruyama (EM) method (see, e.g., \cite{WM08}) to
 the truncated SFDE (\ref{sfde4}) to obtain its numerical solutions. Let
 $k_1$ be a positive integer and set the step size $\D = 1/k_1$. Let
 $t_j=j\D$ for $j=-kk_1, -(kk_1-1), \cdots, -1, 0, 1, \cdots$.  We first need to form discrete-time numerical approximations $X^k_\D(t_j)\approx
 x^k(t_j)$ for $j\ge 0$ given the initial data $x^k_0=\pi_k(\xi)$ and $\theta(0)=i_0$.
 Recall that $\pi_k(\xi)$ depends only on the values
 $\xi(u)$ for $u\in [-k, 0]$. Accordingly, we set $X^k_\D(t_j)=\xi(t_j)$
 for $j= -kk_1, \cdots, -1, 0$ and form $X^k_\D(t_{j+1})$ for $j\ge 0$ by
  \begin{align}\label{DEM}
 X^k_\D(t_{j+1}) =
 X^k_\D(t_j) + F( X^k_\D(t_j), \psi_{1j}, \theta(t_j),t_j) \D
 + G( X^k_\D(t_j), \psi_{2j}, \theta(t_j),t_j)\D B_j,
 \end{align}
where $\D B_j=B(t_{j+1})-B(t_j)$ while
$$
\psi_{1j}
=\sum_{h=-kk_1}^{-1} X^k_\D(t_{j+h})
\mu_8(t_{h}, t_{h+1}) +
 X^k_\D(t_{j-kk_1}) \mu_8(-\infty, -k)
$$
and $\psi_{2j}$ is defined as $\psi_{1j}$  by replacing $\mu_8$ with
$\mu_9$.   Note that
 $\{\theta(t_j)\}_{j\ge 0}$ is a discrete-time Markov chain starting from $\theta(0)=r_0$
 with the one-step transition probability matrix $e^{\D\G}$. The numerical simulation of $\{\theta(t_j)\}_{j\ge 0}$ can be performed in the way as described
 in \cite[p.112]{MY06}.  We next form the continuous-time numerical solution
 \begin{align}\label{CEM1}
 X^k_\D(t) = \sum_{j=0}^\8 X^k_\D(t_j) \II_{[t_j, t_{j+1})}(t), \quad t\ge 0.
 \end{align}
 Although this is the numerical solution we usually compute in practice,
 the numerical analysis is carried out via the
  continuous auxiliary  process defined by
\begin{equation}\label{CEM2}
\bar X^k_{\D}(t):=\xi(0) + \int_0^t F(X^k_{\D}(s), \psi_{1}(s), \bar\o(s), \T(s))\d s + \int_0^t G(X^k_{\D}(s), \psi_{2}(s), \bar\o(s), \T(s))\d B(s),
\end{equation}
for $t\ge 0$ while set $\bar X^k_\D(t)= \xi(t)$ for $t\in [-k, 0]$,  where
$$
\T(t):=t_j,\ \ \  \bar\o(t):=\o(t_j),\ \ \ \psi_{1}(t):= \psi_{1j},\ \ \ \psi_{2}(t):= \psi_{2j},~~~~~t\in[t_j, t_{j+1}).
$$
Note that $\bar X^k_{\D}(t_j)=X^k_{\D}(t_j)$ for all $j$.  That is, $\bar X^k_{\D}(t)$ coincides with the numerical  solution $X^k_{\D}(t)$ at the grid-points.
We need two more assumptions.

\begin{assp}\label{A5.5}
There exist constants $\a\in[1/2, 1]$  and $L_2>0$ such that
$$
|F(x, y, i, t_1)-F(x, y, i, t_2)|\vee|G(x, y, i, t_1)-G(x, y, i, t_2)|\le L_2(1+ |x|+ |y|)|t_1-t_2|^\a
$$
for all $x, y\in \RR^n$, $t_1, t_2 \in \RR_+$ and $i\in\mathbb{S}$.
\end{assp}

\begin{assp}\label{A5.6}
There exist constants $\beta\ge 1/2$ and $L_3>0$  such that the initial function $\xi$ satisfies
$$
  |\xi(s_1)-\xi(s_2)| \le L_3|s_1-s_2|^\beta,\ \ \ \forall s_1, s_2\in (-\8, 0].
$$
\end{assp}

In the remaining of this section we will fix $p\ge 2$, $T >0$ and $k > T$ arbitrarily and let
$C$ stand for a universal positive constant dependent on $p, T,\xi$ etc. but independent of $\D$ and $k$.  Let us present a number of useful lemmas.

\begin{lemma}\label{TH5.8}
Suppose that Assumption \ref{A5.4} holds and $\mu_8, \mu_9\in\CP_r$ (please recall (\ref{psi12}) regarding $\mu_8$ and $\mu_9$).  Then
\begin{equation}\label{EMB}
  \sup_{0<\D\le 1} \mathbb{E}\big(\sup_{t\in[0, T]}|\bar X^k_{\D}(t )|^{p}\big)\le C .
\end{equation}
\end{lemma}

\begin{proof} Fix $\D \in (0,1]$ arbitrarily. It is standard (see, e.g., \cite{M97,MY06}) to show from (\ref{CEM2})
along with Assumption \ref{A5.4} that for $t\in [0,T]$,
\begin{align} \label{M01}
\E\big(\sup_{0\le u\le t} |\bar X^k_{\D}(u)|^{p}\big)
\le C +
C \E\int_0^t ( |X^k_{\D}(s)|^{p} + |\psi_{1}(s)|^{p} + |\psi_{2}(s)|^{p} ) \d s.
\end{align}
For each $s\in [0,T]$, there is a unique $j$ such that $s\in [t_j,t_{j+1})$ and
$\psi_{1}(s)=\psi_{1j}$.   By the definition of $\psi_{1j}$ and $\mu_8 \in\CP_r$, we further derive
\begin{eqnarray}\label{M02}
  |\psi_{1j}| &=& |\sum_{h=-kk_1}^{-1}X^k_\D(t_{j+h})\mu_8(t_h, t_{h+1})+ X^k_\D(t_{j-kk_1})\mu_8(-\infty, -k)| \nonumber\\
  &\le& \sum_{h=-kk_1}^{-1}|X^k_\D(t_{j+h})|\mu_8(t_h, t_{h+1})+ |X^k_\D(t_{j-kk_1})|\mu_8(-\infty, -k) \nonumber \\
  &\le& \sum_{h=-kk_1}^{-1}e^{rt_h}|X^k_\D(t_{j+h})|e^{-rt_h}\mu_8(t_h, t_{h+1})+ e^{rt_{-kk_1}}|X^k_\D(t_{j-kk_1})|e^{-rt_{-kk_1}}\mu_8(-\infty, -k) \nonumber \\
  &\le& \Big(\sup_{h\le0}e^{rt_h}|X^k_\D(t_{j+h})|\Big)\Big(
  \sum_{h=-kk_1}^{-1}e^{-rt_h}\mu_8(t_h, t_{h+1})+e^{-rt_{-kk_1}}\mu_8(-\infty, -k)\Big) \nonumber\\
  &\le&  \Big(\sup_{h\le0}e^{rt_h}|X^k_\D(t_{j+h})|\Big) \int_{-\8}^0 e^{r\D -ru} \mu_8(\d u)
  \nonumber \\
   &\le&  \Big(\sup_{h\le0}e^{rt_h}|X^k_\D(t_{j+h})|\Big)   e^{r}  \mu_8^{(r)} .
\end{eqnarray}
Note that
\begin{eqnarray*}
  &&\sup_{h\le0}e^{rt_h}|X^k_\D(t_{j+h})| 
  =e^{-rt_j}\sup_{h\le j}e^{rt_h}|X^k_\D(t_h)| \\
  &\le& e^{-rt_j}\Big(\sup_{h\le 0}e^{rt_h}|X^k_\D(t_h)|+\sup_{0\le h\le j}e^{rt_h}|X^k_\D(t_h)|\Big) \\
  &\le& e^{-rt_j}\Big(\sup_{\o\le 0}e^{r\o}|\xi(\o)|+\sup_{0\le h\le j}e^{rt_h}|X^k_\D(t_h)|\Big) \\
  &\le& C + \sup_{0\le h\le j}e^{r(t_h-t_j)}|X^k_\D(t_h)|\le C + \sup_{0\le h\le j}|X^k_\D(t_h)|
\end{eqnarray*}
Inserting the above inequality into \eqref{M02} gives
$$
|\psi_{1j}|\le C+ \sup_{0\le h\le j}|X^k_\D(t_h)|.
$$
Consequently
\begin{equation}\label{psi1}
  |\psi_{1}(s)|^{p}\le C\Big(1+ \sup_{0\le u\le s}|\bar X^k_\D(u)|^{p}\Big).
\end{equation}
Similarly
\begin{equation}\label{psi2}
  |\psi_{2}(s)|^{p}\le C\Big(1+ \sup_{0\le u\le s}|\bar X^k_\D(u)|^{p}\Big).
\end{equation}
 Substituting these into (\ref{M01}), we obtain
$$
\E\big(\sup_{0\le u\le t} |\bar X^k_{\D}(u)|^{p}\big)
\le C +
C  \int_0^t  \E\big(\sup_{0\le u\le s} |\bar X^k_{\D}(u)|^{p}\big)\d s.
$$
An application of the Gronwall inequality implies
$$
\E\big(\sup_{0\le u\le T} |\bar X^k_{\D}(u)|^{p}\big) \le C.
$$
As $\D$ is arbitrary, we must have the desired assertion (\ref{EMB}).   $\Box$
\end{proof}

\begin{rmk}\label{rak5.9}
By virtue of Lemma \ref{TH5.8}, it follows from \eqref{psi1} and \eqref{psi2} that
$$
\sup_{t\in [0, T]}\E |\psi_i(t)|^p \le  C,\ \  i=1, 2 .
$$
\end{rmk}

\begin{lemma}\label{L5.9}
Suppose that all conditions  of Lemma \ref{TH5.8} hold. Then for any $\D\in(0, 1]$,
\begin{eqnarray*}
  \E|\bar X^k_\D(t)-X^k_\D(t)|^p &\le& C\D^{p/2},\ \forall\ t\in [0, T]
\end{eqnarray*}
\end{lemma}

\begin{proof}
Fix any $\D\in(0, 1]$.  For each $t\in[0, T]$, there exists a unique integer $j\ge 0$ such that $t_j\le t <t_{j+1}$.  Recalling the definitions of $\bar X^k_\D(\cdot)$ and $X^k_\D(\cdot)$ we derive from (\ref{CEM2})
along with Assumption \ref{A5.4} easily that
\begin{eqnarray*}
  &&\E|\bar X^k_\D(t)-X^k_\D(t)|^p \nonumber \\
  &\le& 2^{p-1}\bigg(\E|\int_{t_j}^tF(X^k_\D(s), \psi_1(s), \bar\o(s), \T(s))\d s|^p+\E|\int_{t_j}^tG(X^k_\D(s), \psi_2(s), \bar\o(s), \T(s))\d B(s)|^p\bigg) \nonumber \\
  &\le& C\D^{(p-1)/2} \int_{t_j}^t \big(1 +   \E|X^k_\D(s)|^p+  \E|\psi_{1}(s)|^p + \E|\psi_{2}(s)|^p\big) \d s.
\end{eqnarray*}
By Lemma \ref{TH5.8} and Remark \ref{rak5.9} we obtain the assertion. $\Box$
\end{proof}

\begin{lemma}\label{T5.10}
Let Assumptions \ref{A5.4}, \ref{A5.5}, \ref{A5.6} hold and $\mu_8, \mu_9\in\CP_r$. Then for any $\D\in(0, 1]$
\begin{eqnarray}\label{5.17}
  \E|x^k(t)-\bar X^k_\D(t)|^2 &\le& C\D,\ \  \forall\ t\in [0, T].
\end{eqnarray}
 \end{lemma}

\begin{proof}
 Fix  $\D\in(0, 1]$ arbitrarily.  Let $e_\D(t)=x^k(t)-\bar X^k_\D(t)$ for $t\in[0, T]$.
 It is straightforward to see that
\begin{eqnarray}\label{5.18}
 \E|e_\D(t)|^2  \le C(I(t)+J(t)),
  \end{eqnarray}
  where
  $$
  I(t) :=  \E \int_0^{t} |F(x^k(s), \psi_1(\pi_k(x^k_s)), \o(s), s)- F(X^k_\D(s), \psi_1(s), \bar\o(s), \T(s))|^2 \d s,
  $$
  $$
J(t) := \E\int_0^{t} |G(x^k(s), \psi_2(\pi_k(x^k_s)), \o(s), s)- G(X^k_\D(s), \psi_2(s), \bar\o(s), \T(s))|^2 \d s.
$$
It is also easy to see that
\begin{eqnarray}\label{eqf1}
 I(t)\le   C (I_1(t)+I_2(t)+I_3(t)),
\end{eqnarray}
where
\begin{eqnarray*}
  I_1(t) &:=& \displaystyle  \E \int_0^{t} |F(x^k(s), \psi_1(\pi_k(x^k_s)), \o(s), s)-  F(X^k_\D(s), \psi_1(s),  \o(s), s)|^2   \d s,\\
I_2(t) &:=&  \displaystyle \E \int_0^{t} |   F(X^k_\D(s), \psi_1(s),  \o(s), s)- F(X^k_\D(s), \psi_1(s), \bar\o(s), s)|^2 \d s,\\
I_3(t) &:=& \displaystyle  \E \int_0^{t} |    F(X^k_\D(s), \psi_1(s), \bar\o(s), s)- F(X^k_\D(s), \psi_1(s), \bar\o(s), \tau(s))|^2 \d s  .
\end{eqnarray*}
By Assumption \ref{A5.4} and Lemma \ref{L5.9}, we have
\begin{eqnarray}\label{I1}
   I_1(t)
  &\le & \displaystyle 2 L_1^2\E\int_0^{t} \bigg(| x^k(s)-  X^k_\D(s)|^2+ |\psi_1(\pi_k(x^k_s))-\psi_1(s)|^2 \bigg) \d s\nonumber \\
  &\le & C\D +C\displaystyle \E\int_0^{t}  |e_\D(s) |^2\d s +C\displaystyle \E\int_0^{t} |\psi_1(\pi_k(x^k_s))-\psi_1(s)|^2   \d s .
\end{eqnarray}
 For each $t>0$, let $N=\lfloor t/\D\rfloor$  and $t_{N+1}=t$ for a meanwhile.
Recalling the definition of $\pi_k$, using  the H$\ddot{\hbox{o}}$lder inequality and Assumption \ref{A5.6}, we derive that
\begin{eqnarray}
  &&\E\int_0^{t} |\psi_1(\pi_k(x^k_s))-\psi_1(s)|^2 \d s
  = \sum_{j=0}^{N} \E\int_{t_j}^{t_{j+1}} |\psi_1(\pi_k(x^k_s))-\psi_{1j}|^2 \d s  \nonumber\\
  &=&  \sum_{j=0}^{N} \E\int_{t_j}^{t_{j+1} } \big|\int_{-k}^0 x^k(s+\o)\mu_8(\d\o)+ x^k(s-k)\mu_8(-\infty, -k)  \nonumber\\
  &&-\sum_{h=-kk_1}^{-1}X^k_\D(t_{j+h})\mu_8(t_h, t_{h+1})- X^k_\D(t_j-k)\mu_8(-\infty, -k)\big|^2 \d s  \nonumber
  \end{eqnarray}
\begin{eqnarray}\label{5.21}
  &\le& \sum_{j=0}^N \E\int_{t_j}^{t_{j+1}}\bigg(\sum_{h=-kk_1}^{-1} \int_{t_h}^{t_{h+1}}  |x^k(s+\o)-X^k_\D(t_j+t_h)|  \mu_8(\d\o) \nonumber\\
  &&+ |x^k(s-k)- X^k_\D(t_j-k)|  \mu_8(-\infty, -k)\bigg)^2\d s \nonumber\\
  &\le& 2 \sum_{j=0}^N \E\int_{t_j}^{t_{j+1}}\bigg(\sum_{h=-kk_1}^{-1} \int_{t_h}^{t_{h+1}}  |x^k(s+\o)-X^k_\D(t_j+\o)|  \mu_8(\d\o)\bigg)^2\d s \nonumber\\
  &&+2 \sum_{j=0}^N \E\int_{t_j}^{t_{j+1}}\bigg( |\xi(s-k)- \xi(t_j-k)|  \mu_8(-\infty, -k)\bigg)^2\d s \nonumber\\
 &\le& 2 \sum_{j=0}^N \E\int_{t_j}^{t_{j+1}} \int_{-k}^0 |x^k (s+\o)- X^k_\D(t_j+\o)|^2 \mu_8(\d\o)\d s+ 2 TL_3^2\D^{2\beta}\nonumber\\
 &\le& 4 \sum_{j=0}^N \E\int_{t_j}^{t_{j+1}} \int_{-k}^{-s} |\xi (s+\o)- \xi(t_j+\o)|^2 \mu_8(\d\o)\d s+2 TL_3^2\D^{2\beta }\nonumber\\
 &&+4 \sum_{j=0}^N \E\int_{t_j}^{t_{j+1}} \int_{-s}^{0} |x^k (s+\o)- X^k_\D(t_j+\o)|^2 \mu_8(\d\o)\d s\nonumber\\
 &\le& C\D^{2\beta}+4   \E\int_{0}^{t } \int_{-s}^{0} |x^k (s+\o)- X^k_\D(\tau(s)+\o)|^2 \mu_8(\d\o)\d s.
\end{eqnarray}
But we obviously have
\begin{eqnarray*}
 &&  \E \int_{0}^{t } \int_{-s}^{0} |x^k (s+\o)- X^k_\D(\tau(s)+\o)|^2 \mu_8(\d\o)\d s\\
& \leq &    3\E \int_{0}^{t } \int_{-s}^{0}|e_\D(s+\o) |^2\mu_8(\d\o)\d s
+ 3 \E \int_{0}^{t } \int_{-s}^{0} |\bar X^k_\D(s+\o)-\bar X^k_\D(\tau(s)+\o)|^2  \mu_8(\d\o)\d s\\
 &&  +   3 \E \int_{0}^{t } \int_{-s}^{0}|\bar X^k_\D(\tau(s)+\o)- X^k_\D(\tau(s)+\o)|^2\mu_8(\d\o)\d s.
\end{eqnarray*}
In the same way Lemma \ref{L5.9} was proved,  we can show that $\E|\bar X^k_\D(s+\o)-\bar X^k_\D(\tau(s) +\o)|^2 \le C\D $ for any $s\in[0, T]$.  Applying this and Lemma \ref{L5.9} to the inequality above yields
\begin{eqnarray*}
 &&\displaystyle \E \int_{0}^{t } \int_{-s}^{0} |x^k (s+\o)- X^k_\D(\tau(s)+\o)|^2 \mu_8(\d\o)\d s \nonumber\\
  &\le &   C\D +\displaystyle \E \int_{0}^{t } \int_{-s}^{0}|e_\D(s+\o) |^2\mu_8(\d\o)\d s\nonumber\\
  &\le & C\D +\displaystyle \E \int_{-\infty}^{0 } \int_{0}^{t}|e_\D(s ) |^2\d s\mu_8(\d\o) \nonumber\\
  &= & C\D +\displaystyle \E  \int_{0}^{t}|e_\D(s ) |^2\d s  .
\end{eqnarray*}
 Substituting this into \eqref{5.21} yields
\begin{eqnarray*}
 \E\int_0^{t} |\psi_1(\pi_k(x^k_s))-\psi_1(s)|^2 \d s &\le& C\D + C\E\int_{0}^{t} |e_\D(s)|^2 \d s.
\end{eqnarray*}
Consequently, inserting this into \eqref{I1} we arrive  at
\begin{eqnarray}\label{eqf2}
  I_1(t) \le C\D   + C\E\int_{0}^{t} |e_\D(s)|^2\d s.
\end{eqnarray}
 By the Markov property of $\o(\cdot)$, Assumption \ref{A5.4}, ,Lemma \ref{TH5.8} and Remark \ref{rak5.9}, we derive that
\begin{eqnarray}\label{eqf3}
  I_{2}(t) &=& \sum_{j=0}^{N} \E\int_{t_j}^{t_{j+1}} |F(X^k_\D(t_j), \psi_{1j}, \o(s), t_j)-F(X^k_\D(t_j), \psi_{1j}, \o(t_j), t_j)|^2 \d s  \nonumber\\
  &=& \sum_{j=0}^N \E\int_{t_j}^{t_{j+1}} |F(X^k_\D(t_j), \psi_{1j}, \o(s), t_j)-F(X^k_\D(t_j), \psi_{1j}, \o(t_j), t_j)|^2 \II_{\{\o(s)\neq\o(t_j)\}}  \d s  \nonumber\\
  &\le& C\sum_{j=0}^N \int_{t_j}^{t_{j+1}} \E\bigg[\big(1+ |X^k_\D(t_j)|^2+ |\psi_{1j}|^2\big)\E\big(\II_{\{\o(s)\neq\o(t_j)\}}|
\mathcal{F}_{t_j}\big)\bigg]  \d s  \nonumber\\
  &\le& C\D\sum_{j=0}^N \int_{t_j}^{t_{j+1}} \E\bigg[\big(1+ |X^k_\D(t_j)|^2+ |\psi_{1j}|^2\big)\bigg]  \d s  \nonumber\\
  &\le& C\D.
\end{eqnarray}
It follows from  Assumption \ref{A5.5}, Lemma \ref{TH5.8} and Remark \ref{rak5.9} that
\begin{eqnarray}\label{eqf4}
  I_3(t) &\le&  C\E\int_0^{t} \Big(1+ |\bar X^k_\D(s)|^2+ |\psi_1(s)|^2\Big)\D^{2\a}\d s  \le  C\D^{2\a} .
\end{eqnarray}
Inserting \eqref{eqf2}-\eqref{eqf4} into \eqref{eqf1}, we obtain that $I(t)\le C\D+  C\int_0^t \E|e_\D(s)|^2\d s$.  Similarly, we can show  $J(t)\le C\D+  C\int_0^t \E|e_\D(s)|^2\d s$.
Putting these into (\ref{5.18}) gives
\begin{eqnarray*}
  \E|e_\D(t)|^2 &\le& C\D+  C\int_0^t \E|e_\D(s)|^2\d s.
\end{eqnarray*}
An application of the Gronwall inequality yields that
$$
  \E|e_\D(t)|^2  \le  C\D
$$
as required.  The proof is hence complete. $\Box$
\end{proof}

Combing Lemmas \ref{L5.9} and  \ref{T5.10},  we obtain the strong convergence of  the EM numerical solutions to the true solution of the truncated  SFDE (\ref{sfde4}).

\begin{theorem}
Let Assumptions \ref{A5.4}, \ref{A5.5}, \ref{A5.6} hold and $\mu_8, \mu_9\in\CP_r$. Then for any $\D\in(0, 1]$
\begin{eqnarray}
  \E|x^k(t)-X^k_\D(t)|^2 &\le& C\D,\  \  \forall\ t\in [0, T].
\end{eqnarray}
\end{theorem}

  On the other hand,  in a similar way as  Theorem \ref{T3.5} was proved, we can show the following corollary}.

\begin{coro}
Suppose that  Assumption \ref{A5.4} holds and $\mu_8, \mu_9\in \CP_b$ with $b>r$. Then the solution $x^k(t)$ of  the truncated  SFDE (\ref{sfde4}) approximates the solution $x(t)$ of the given SFDE (\ref{sfde3}) in the sense that
\begin{equation}\label{CR}
  \E|x(t)-x^k(t)|^p\le Ce^{-(b-r) pk},\ \forall\ t\in [0, T].
\end{equation}
\end{coro}

Consequently, we obtain the following strong convergence result of  the EM numerical solutions to the true solution  of the given SFDE (\ref{sfde3}).

\begin{theorem}
Suppose that Assumptions \ref{A5.4}, \ref{A5.5}, \ref{A5.6} hold and $\mu_8, \mu_9\in \CP_b$ with $b>r$. Then for any $\D \in(0, 1]$ and any integer $k > T$,
\begin{equation}
  \E|x(t)-X^k_\D(t)|^2\le C(e^{-2(b-r)  k}+ \D),\ \forall\ t\in [0, T].
\end{equation}
\end{theorem}

\subsection{Highly nonlinear case}

In Section \ref{SecHN}, we already showed that the solution $x^k(t)$ of the truncated SFDE (\ref{sfde4})
approximates the solution of the given SFDE (\ref{sfde3}) under Assumptions \ref{A4.5} and \ref{A4.6}.  To obtain the numerical solution of  the truncated SFDE (\ref{sfde4}) under these assumptions, we can apply the truncated EM method (see, e.g., \cite{GMY18,ZSL18}).  Due to the page limit, we leave the details to the reader but discuss a couple of examples and carry out some simulations using MATLAB to illustrate the idea.

\begin{expl}\label{eg5.4} {\rm
In 1973, the well-known Black-Scholes model with constant volatility was presented. But tests of this model on real market data have questioned the assumption of constant volatility in the stock dynamics. For this reason, several variants of the Black-Scholes model with non-constant volatility such as stochastic functional volatility equations  have been proposed \cite{SM17}. In this example, we consider a scalar stochastic functional volatility equation with infinite delay of the form
\begin{align}\label{eq5.5}
\d x(t) = f(x_t, \o(t), t)\d t + g(x_t, \o(t), t)\d B(t),\ \ \ t\ge 0,
\end{align}
where the coefficients are defined by
$$
f(\f, i, t) = \left\{
\begin{array}{cc}
1+4\f(0)-4\f^3(0),\ \ \  i=1, \\
2+3\f(0)-5\f^3(0),\ \ \  i=2, \\
\end{array}\right.
$$
$$
g(\f, i, t) = \left\{
\begin{array}{cc}
\int_{-\infty}^0 \f^2(u) \mu(\d u),\ \ \ i=1, \\
\frac{1}{2}\int_{-\infty}^0 \f^2(u) \mu(\d u),\ \ \ i=2, \\
\end{array}\right.
$$
for $\f\in\CC_{1/5}$,    the probability measure $\mu(\cdot)$ has its probability density function $e^u$ on $(-\infty, 0]$ (i.e., $\mu(\d u)=e^u\d u$), $B(t)$ is a scalar Brownian motion and $\o(t)$ is a Markov chains on the state space $\mathbb{S}=\{1, 2\}$ with its generator
$$
\G=\left(
\begin{array}{ccc}
-1  &  1 \\
2   &  -2 \\
\end{array}\right).
$$
Let the initial data $\xi(u)=e^{u}\in\CC_{1/5}$ and $\o(0)=1$.
Recalling the definition of truncation mapping $\pi_k$ we get the corresponding approximation SFDEs
\begin{align}\label{eq5.6}
\d x^k(t) = f_k(x^k_t, \o(t), t)\d t + g_k(x^k_t, \o(t), t)\d B(t).
\end{align}
 Here $f_k$ and $g_k$ are defined by
$$
f_k(\f, i, t) = \left\{
\begin{array}{cc}
1+4\f(0)-4\f^3(0),\ \ \ i=1, \\
2+3\f(0)-5\f^3(0),\ \ \ i=2,
\end{array}\right.
$$
and
$$
g_k(\f, i, t) = \left\{
\begin{array}{cc}
e^{-k}\f^2(-k)+\int_{-k}^0\f^2(u)\mu(\d u),\ \ \ i=1, \\
\frac{1}{2}e^{-k}\f^2(-k)+\frac{1}{2}\int_{-k}^0\f^2(u)\mu(\d u),\ \ \ i=2, \\
\end{array}\right.
$$
for $\f\in\CC_{1/5}$.
One observes that $f$ and $g$ satisfy Assumption \ref{A2.1} owing to $\mu\in\mathcal{P}_\g$ for any $\g\in[0, 1)$ and  Assumption \ref{A2.2} with $W_1(x, t)=V(x, i, t)=x^4$ and $W_2(x, t)= x^6$. Thus, \eqref{eq5.5} has a unique global solution.  Furthermore, Assumption \ref{A3.4} holds with $\bar q=2$, $\bar V(x,t)=x^2$, $\mu_4=\de_0\in\mathcal{P}_\g$ for any $\g>0$, $ \mu_5=\mu$ and $U(x,y)=|x-y|^2|x+y|^2\in\mathcal{U}_{0,4}$. By virtue of Theorem \ref{T3.5}, $x^k(T)$ converges to $x(T)$  exponentially  for any $T>0$.

In order to test the efficiency of the result in Theorem \ref{T3.5} we carry out  some numerical experiments by MATLAB.
We use the truncated EM numerical solution of \eqref{eq5.6} with $k=200$ and $\D=2^{-6}$ as the exact solution of \eqref{eq5.5} and plot the mean square error $\E|x(10)-x^k(10)|^2$ for 1000 sample points between the solution of \eqref{eq5.5} and that of \eqref{eq5.6} as function of $k$ when $k\in\{10,11,12,13,14\}$.
 Furthermore, in order to characterize the error between the exact solution and numerical solution with respect to $\D$, we take the truncated EM numerical solution of \eqref{eq5.6} with $k=200$ and $\D=2^{-18}$ as the exact solution of \eqref{eq5.5}.  Figure 1 depicts  the root mean square error $(\E|x(10)-X^k_\D(10)|^2)^{1/2}$ between the exact solution and the numerical solution of \eqref{eq5.5} with $k=50$, as a function of $\D\in \{2^{-6}, 2^{-8}, 2^{-10}, 2^{-12}, 2^{-14}\}$ for 1000 sample points.

\begin{figure}[h]\label{fig1}
  \centering
  \includegraphics[width=15cm]{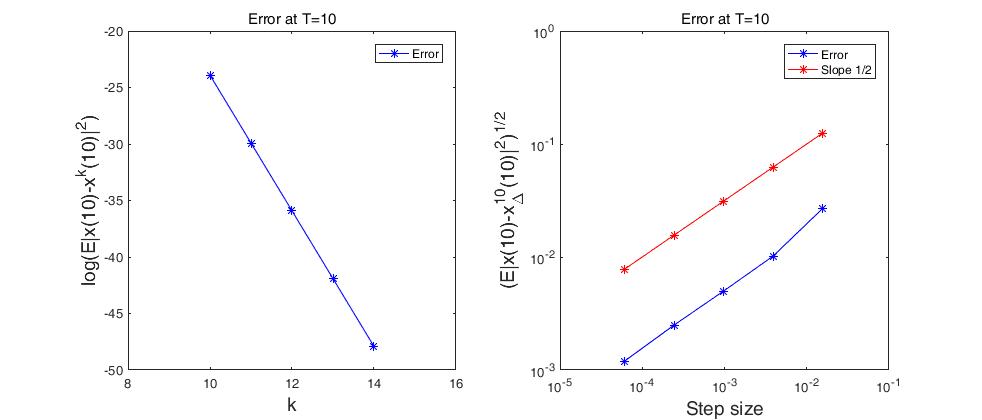}\\
  \caption{(a) The mean square error for 1000 sample points between  $x(10)$  and $x^k(10)$ as the function  of  $k\in\{10,11,12,13,14\}$.  (b) The root mean square errors for 1000 sample points between $x(10)$ and $X^{50}_\D(10)$ as the function of  $\D \in\{2^{-6}, 2^{-8}, 2^{-10}, 2^{-12}, 2^{-14}\}$.  }
\end{figure}
}
\end{expl}

\begin{expl}\label{eg5.5}{\rm
The delay Lotka-Volterra systems  have received great attention owing to their extensive applications. The author of \cite{Liu15} analyzed the global asymptotic stability of the generic  stochastic Lokta-Volterra systems with infinite delay. Let us consider such a system with special coefficients  described by
\begin{equation}\label{eq5.7}
\left\{
\begin{aligned}
\d x_1(t) = &x_1(t)[(0.5+0.1\sin t)-0.8x_1(t)-0.2x_2(t)]\d t +
            0.5x_1(t)\d B_1(t), \\
\d x_2(t) = &x_2(t)\big[(0.3+0.2\sin 2t)-0.6x_2(t) \\ &-0.12\int_{-\infty}^0 x_1(t+u)\mu(\d u))\big]\d t + 0.5x_2(t)\d B_2(t),\\
\end{aligned}\right.
\end{equation}
and the initial data are given by $\xi_1(u)=0.8e^u$ for $u\le 0$ and $x_2(0)=0.6$.  Here $(B_1(t), B_2(t))$ is a 2-dimensional Brownian motion and $\mu$ is the same probability measure  as  in Example \ref{eg5.4}. According to \cite{Liu15}, \eqref{eq5.7} has a unique global positive solution $x(t)=(x_1(t), x_2(t))$ for $t\geq 0$.  One observes that  Assumption  \ref{A3.2} holds with $\beta(x)=W_1(x, t)= V(x, i, t)=|x|^4, W_2(x, t)=0$ and $\mu_1=\delta_0$. Assumption \ref{A3.3} holds with $\mu_3(\cdot) =0.5(\delta_0(\cdot) + \mu(\cdot))$. For any $k>0$,   the corresponding approximation equation is
\begin{equation}\label{eq5.8}
\left\{
\begin{aligned}
\d x^k_1(t) = &x^k_1(t)[(0.5+0.1\sin t)-0.8x^k_1(t)-0.2x^k_2(t)]\d t +
            0.5x^k_1(t)\d B_1(t), \\
\d x^k_2(t) = &x^k_2(t)\big[(0.3+0.2\sin2t)-0.6x^k_2(t)-0.12e^{-k}x_1(t-k)  \\
        &-0.12\int_{-k}^0 x^k_1(t+u)\mu(\d u))\big]\d t + 0.5x^k_2(t)\d B_2(t),
\end{aligned}
\right.
\end{equation}
 which has a unique global solution $x^k(t)=(x^k_1(t), x^k_2(t))$ for $t\geq 0$.  
 Thus, by virtue of  Theorem \ref{T3.3}, $x^k(t)$ converges to the solution $x(t)$ in $L^2$. For illustration, we carry out some numerical experiments using MATLAB.  Due to the unsolvability of \eqref{eq5.8} we  regard  the numerical solution of the truncated EM scheme with $\D=2^{-6}$ and $k=200$ as the exact $x(t)$ of \eqref{eq5.7}, while 
for $k\in\{10,12,14,16,18\}$, we regard the numerical solution of the  truncated EM scheme with $\D=2^{-6}$ as the exact  $x^k(t)$ of \eqref{eq5.8}.
Furthermore, we view the truncated EM numerical solution of \eqref{eq5.8} with $k=200$ and $\D=2^{-18}$ as the exact solution of \eqref{eq5.7}. Let $k=30$ and $\D\in \{2^{-6}, 2^{-8}, 2^{-10}, 2^{-12}, 2^{-14}\}$. Figure 2 depicts the root mean square error $(\E|x(10)-X^k_\D(10)|^2)^{1/2}$ between the exact solution and the numerical solution of \eqref{eq5.7} for 1000 sample points, as functions of $\D$ for 1000 sample points.

\begin{figure}
  \centering
  \includegraphics[width=15cm]{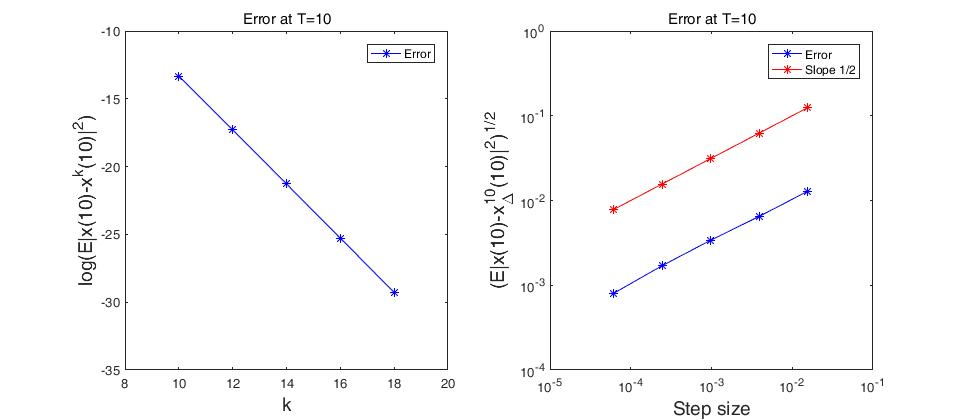}\\
  \caption{(a) The mean square error for 1000 sample points between   $x(10)$  and $x^k(10)$ as the function  of  $k\in\{10,12,14,16,18\}$.  (b) The root mean square error for 1000 sample points between   $x(10)$  and $X_\Delta^{30}(10)$ as the function  of  $\Delta\in\{2^{-6}, 2^{-8}, 2^{-10}, 2^{-12}, 2^{-14}\}$. }
\end{figure}
}
\end{expl}

\section*{Acknowledgements}

The authors would like to thank the National Natural Science Foundation of China (No. 11971096), the National Key R$\&$D Program of China (2020YFA0714102), the Natural Science Foundation of Jilin Province (No. YDZJ202101ZYTS154),    the Fundamental Research Funds for the Central Universities,
the Royal Society (WM160014, Royal Society Wolfson Research Merit Award),
the Royal Society of Edinburgh (RSE1832)
 for their financial support.

\bigskip
\noindent

\end{document}